\documentclass{article}
\usepackage{graphicx} 
\usepackage{tikz}
\usepackage{mathrsfs}
\usepackage{amsmath,amsthm,color,eepic}
\usepackage{amssymb} \usepackage{verbatim}

\newtheorem{theorem}{Theorem}

\newtheorem{lemma}{Lemma}

\theoremstyle{definition}

\newtheorem{claim}{Claim}
\newtheorem{observation}{Observation}

\newtheorem{case}{Case}

\begin{document}

\title{Characterizing forbidden pairs for spanning $\varTheta$-subgraphs of 2-connected graphs\thanks{Supported by NSFC (Nos. 12171393, 12001242) and Shaanxi Fundamental Science Research
Project for Mathematics and Physics (No. 22JSZ009).}}
\author{Binlong Li$^{a,b}$,~Ziqing Sang$^{a,b}$,~Shipeng Wang$^{c,}$\thanks{Corresponding author. E-mail: binlongli@nwpu.edu.cn (B. Li), ziqing\_sang@mail.nwpu.edu.cn (Z. Sang),
spwang22@ujs.edu.cn (S. Wang)}~\\[2mm]
{\small $^a$School of Mathematics and Statistics,}\\[-0.8ex]
{\small Northwestern Polytechnical University, Xi'an, Shaanxi 710072, China}\\
{\small $^b$Xi'an-Budapest Joint Research Center for Combinatorics,}\\[-0.8ex]
{\small Northwestern Polytechnical University, Xi'an, Shaanxi 710072, China}\\
{\small $^c$Department of Mathematics,}\\[-0.8ex]
{\small Jiangsu University, Zhenjiang, Jiangsu 212013, China}\\
}
\date{}

\maketitle

\begin{abstract}
Let $\mathcal{F}$ be a set of connected graphs, and let $G$ be a graph. We say that $G$ is \emph{$\mathcal{F}$-free} if it does not contain $F$ as an induced subgraph for all $F\in\mathcal{F}$, and we call $\mathcal{F}$ a forbidden pair if $|\mathcal{F}|=2$. A \emph{$\varTheta$-graph} is the graph consisting of three internally disjoint paths with the same pair of end-vertices. If the $\varTheta$-subgraph $T$ contains all vertices of $G$, then we call $T$ a \emph{spanning $\varTheta$-subgraph} of $G$. In this paper, we characterize all pairs of connected graphs $R,S$ such that every 2-connected	$\{R,S\}$-free graph has a spanning $\varTheta$-subgraph. In order to obtain this result, we also characterize all minimal 2-connected non-cycle claw-free graphs without spanning $\varTheta$-subgraphs.
\smallskip

\end{abstract}

\section{Introduction}
We follow the most common graph-theoretic terminology and notation, and for concepts not defined here we refer to Bondy and Murty \cite{2}.

Let $G$ be a graph, $v \in V(G)$, $S\subseteq V(G)$ and $H$ be a subgraph of $G$. 
We use $N_{G}(v)$ and $N_{H}(v)$ to denote the set of neighbors of $v$ in $G$ and $H$, respectively. 
We use $G-H$ to denote the subgraph
induced by $V(G)\setminus V(H) $.
Let $N_{S}(H)=\bigcup_{v\in V(H)}N_{G}(v)\cap S$,
and let
$\langle S\rangle_{G}$ denote the subgraph of $G$ induced by $S$.
We often omit the subscript when the underlying graph is clear.
Let $\mathcal{F}$ be a set of connected graphs. A graph $G$ is said to be \emph{$\mathcal{F}$-free} if $G$ does not contain $F$ as an induced subgraph for all $F\in\mathcal{F}$, and we call $F$ a \emph{forbidden subgraph} of $G$ and
$\mathcal{F}$ a forbidden pair if $|\mathcal{F}|=2$. If $G$ is $K_{1,3}$-free then we say that $G$ is \emph{claw-free}. Let $N_{i,j,k}$ be the graph obtained by attaching three vertex-disjoint paths of lengths $i,j,k\geq 0$ to a triangle (so $K_3=N_{0,0,0}$). In the special case when $i,j\geq 1$ and $k=0$ (or $i\geq 1$ and $j=k=0$), $N_{i,j,k}$ is also denoted $B_{i,j}$ (or $Z_{i}$, respectively). We present in Figure~\ref{Nijk} several graphs that play important roles in this paper. 
A cycle (or a path) in a graph $G$ is called \emph{Hamilton cycle} (or \emph{Hamilton path})  if it contains all vertices of $G$. We say that $G$ is \emph{hamiltonian} if it contains a Hamilton cycle. For a path $P$ in a graph $G$, we use $|P|$ to denote the length of $P$.

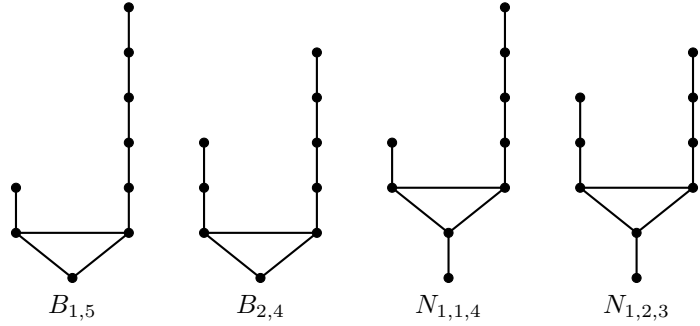
\begin{figure}[h]
\centering
\begin{tikzpicture}[scale=0.1]

\begin{scope}
\draw[fill=black] (0,0) {coordinate (x1)} circle (0.6);
\foreach \x in {1,2} \draw[fill=black] (-7.5,\x*6) circle (0.6);
\coordinate (y1) at (-7.5,6); \coordinate (y2) at (-7.5,12);
\foreach \x in {1,2,...,6} \draw[fill=black] (7.5,\x*6) circle (0.6);
\coordinate (z1) at (7.5,6); \coordinate (z2) at (7.5,36);
\draw[thick] (x1)--(y1)--(z1)--(x1);
\draw[thick] (y1)--(y2) (z1)--(z2);
\node[below] at (0,-1) {$B_{1,5}$};
\end{scope}

\begin{scope}[xshift=25cm]
\draw[fill=black] (0,0) {coordinate (x1)} circle (0.6);
\foreach \x in {1,2,3} \draw[fill=black] (-7.5,\x*6) circle (0.6);
\coordinate (y1) at (-7.5,6); \coordinate (y2) at (-7.5,18);
\foreach \x in {1,2,...,5} \draw[fill=black] (7.5,\x*6) circle (0.6);
\coordinate (z1) at (7.5,6); \coordinate (z2) at (7.5,30);
\draw[thick] (x1)--(y1)--(z1)--(x1);
\draw[thick] (y1)--(y2) (z1)--(z2);
\node[below] at (0,-1) {$B_{2,4}$};
\end{scope}

\begin{scope}[xshift=50cm]
\foreach \x in {1,2} \draw[fill=black] (0,12-\x*6) circle (0.6);
\coordinate (x1) at (0,6); \coordinate (x2) at (0,0);
\foreach \x in {1,2} \draw[fill=black] (-7.5,6+\x*6) circle (0.6);
\coordinate (y1) at (-7.5,12); \coordinate (y2) at (-7.5,18);
\foreach \x in {1,2,...,5} \draw[fill=black] (7.5,6+\x*6) circle (0.6);
\coordinate (z1) at (7.5,12); \coordinate (z2) at (7.5,36);
\draw[thick] (x1)--(y1)--(z1)--(x1);
\draw[thick] (x1)--(x2) (y1)--(y2) (z1)--(z2);
\node[below] at (0,-1) {$N_{1,1,4}$};
\end{scope}

\begin{scope}[xshift=75cm]
\foreach \x in {1,2} \draw[fill=black] (0,12-\x*6) circle (0.6);
\coordinate (x1) at (0,6); \coordinate (x2) at (0,0);
\foreach \x in {1,2,3} \draw[fill=black] (-7.5,6+\x*6) circle (0.6);
\coordinate (y1) at (-7.5,12); \coordinate (y2) at (-7.5,24);
\foreach \x in {1,2,...,4} \draw[fill=black] (7.5,6+\x*6) circle (0.6);
\coordinate (z1) at (7.5,12); \coordinate (z2) at (7.5,30);
\draw[thick] (x1)--(y1)--(z1)--(x1);
\draw[thick] (x1)--(x2) (y1)--(y2) (z1)--(z2);
\node[below] at (0,-1) {$N_{1,2,3}$};
\end{scope}

\end{tikzpicture}
\caption{Graphs $B_{1,5}$, $B_{2,4}$, $N_{1,1,4}$ and $N_{1,2,3}$.}\label{Nijk}
\end{figure}

The first characterization of forbidden pairs of connected subgraphs for hamiltonicity of 2-connected graphs was given by Bedrossian in \cite{91}. Faudree and Gould refined it and gave a complete characterization of forbidden pairs for hamiltonicity of 2-connected graphs with order at least 10. 

\begin{theorem}[Faudree and Gould \cite{5}] \label{large order}
 The only connected graph $R$ of order at least $3$ such that every $2$-connected $R$-free graph is hamiltonian, is $P_3$. Let $R,S$ be a pair of connected graphs of order at least $3$ with $R,S\neq P_3$. Then every $2$-connected $\{R,S\}$-free graph of order at least $10$ is hamiltonian, if and only if (up to symmetry), $R=K_{1,3}$ and $S$ is an induced subgraph of $P_{6}$, $B_{1,2}$, $N_{1,1,1}$ or $Z_3$.
\end{theorem}

Since then, the characterizations of forbidden pairs for graph properties have been studied by many researchers. For a property $\mathcal{P}$, it is a popular research topic to give forbidden induced subgraph conditions forcing a graph to have the property $\mathcal{P}$, see, e.g. for hamiltonian properties refer to~\cite{91,5,LiVr}, for perfect matchings refer to \cite{FuFuPlSaSc,FuKaLuOtPlSa}, for 2-factors refer to~\cite{AlFuSa,AlFuSa2,FaFaRy,HoRyVrWaXi}, for supereulerianity refer to~\cite{llw,LvXi,LvXi2}. In general, characterizing those forbidden induced subgraphs $F$ for a graph to have some property $\mathcal{P}$ is not very difficult. However, characterizing the forbidden sets $\mathcal{F}$ ($|\mathcal{F}|\geq 2$)  for a graph to have some property $\mathcal{P}$ is not easy and then it will be interesting.

A \emph{$\varTheta$-graph} $T$ is the graph consisting of three internally disjoint paths 
with the same pair of end-vertices. 
If a $\varTheta$-subgraph $T$ contains all vertcies of $G$, then we call it a \emph{spanning $\varTheta$-subgraph} of $G$. We notice that if a graph $G$ has a Hamilton cycle, then $G$ has a spanning $\varTheta$-subgraph (unless $G$ itself is a cycle); and if $G$ has a spanning $\varTheta$-subgraph, then $G$ has a Hamilton path.

Continue this line of investigation, in this paper we consider the property $\mathcal{P}$ with respect to the existence of spanning $\varTheta$-subgraphs. We first characterize all graphs $R$ such that every 2-connected $R$-free non-cycle graph has a spanning $\varTheta$-subgraph.

\begin{observation}\label{Th1}
 Let $R$ be a connected graph of order at least $3$. Then every $2$-connected $R$-free non-cycle graph has a spanning $\varTheta$-subgraph, if and only if $R=P_3$.
\end{observation}

We characterize all pairs $R,S$ of connected graphs such that every 2-connected $\{R,S\}$-free non-cycle graph has a spanning $\varTheta$-subgraph, which is our first main result. In order to avoid trivial cases, we suppose that neither $R$ nor $S$ is $P_{3}$ by  virtue of Observation \ref{Th1}.

\begin{theorem}\label{Th2}
 Let $R,S$ be a pair of connected graphs of order at least $3$, $R, S\neq P_3$. Then every 2-connected $\{R,S\}$-free non-cycle graph has a spanning $\varTheta$-subgraph, if and only if (up to symmetry), $R=K_{1,4}$ and $S=P_4$, or $R=K_{1,3}$ and $S$ is an induced subgraph of $B_ {1,5}, B_{2,4}, N_{1,1,4}$ or $N_{1,2,3}$.
\end{theorem}

This paper is organized as follows: In Section 2, we will give the characterization of minimal graphs with property `to be 2-connected, non-cycle, claw-free and without spanning $\varTheta$-subgraphs'. In section 3, we give some lemmas that is useful in our proofs. In Section 4, we give the proof of Theorem~\ref{Th3}
(our second main result)
given in Section 2. In Section 5, we give the proof of Observation~\ref{Th1} and Theorem~\ref{Th2}.

%

\section{Minimal claw-free graphs without spanning $\varTheta$-subgraphs}

We say that a graph $G$ is \emph{minimal}  with respect to a property $\mathcal{P}$ if $G$ satisfies $\mathcal{P}$ and there exists no proper induced subgraphs of $G$ satisfying property $\mathcal{P}$. Brousek~\cite{B1} characterized all minimal graphs with respect to the property `to be 2-connected, claw-free and non-hamiltonian'.

Define the graph $P_{k_1,k_2,k_3}$, where $k_i\geq 2$, as the graph obtained by taking two vertex-disjoint triangles $a_1a_2a_3a_1$, $b_1b_2b_3b_1$, and by joining every pair of vertices $\{a_i,b_i\}$ ($i=1,2,3$) by a triangle $a_ib_ic_ia_i$ for $k_i=2$, or by a path $a_ic_i^1\ldots c_i^{k_i-2}b_i$ for $k_i\geq 3$.

\begin{theorem}[Brousek \cite{B1}]
Every $2$-connected non-hamiltonian claw-free graph contains an induced subgraph isomorphic to $P_{k_1,k_2,k_3}$, for some $k_i\geq 2$, $i=1,2,3$.
\end{theorem}

Motivated by this and in order to prove Theorem~\ref{Th2}, we characterize all minimal graphs with respect to the property `to be 2-connected, non-cycle, claw-free and without spanning $\varTheta$-subgraphs'. In the following, we first give some basic definitions of minimal graphs.

\subsection{Links and chains}

By a \emph{link} we mean a graph with two labeled vertices. If $L$ is a link with labeled vertices $x$ and $y$, then $x$ is the origin, $y$ is the terminus of the link, and $x,y$ are the end-vertices of the link. We also say $L$ is a link from $x$ to $y$, or connecting $x$ and $y$. For convenient, we will denote a link $L$ from $x$ to $y$ by $L(x,y)$. If $L(x,y)$ is a triangle, or a path connecting $x,y$ of length at least 2, then we call it a \emph{pure link}.

We define two classes of links $\mathcal{L}_1$, $\mathcal{L}_2$, as follows, see Figure \ref{L12}. We remark that in all figures of the paper, a red dotted line stands for a pure link. By a \emph{clique} we mean a complete subgraph. Let
\begin{itemize}
\item $\mathcal{L}_1(x,y)$ be the family of links from $x$ to $y$ consisting of a 5-clique $\langle \{x,y,a,b,c\}\rangle$ and a triangle $\langle \{a',b',c'\}\rangle$, and three pure links from $a$ to $a'$, from $b$ to $b'$, and from $c$ to $c'$, respectively;
\item $\mathcal{L}_2(x,y)$ be the family of links from $x$ to $y$ consisting of two triangles $\langle \{x,a,b\}\rangle$ and $\langle \{y,a',b'\}\rangle$, and two pure links from $a$ to $a'$, and from $b$ to $b'$, respectively.
\end{itemize}

\begin{figure}[h]
\centering
\begin{tikzpicture}[scale=0.08]

\begin{scope}
\draw[fill=black] (-60,20){coordinate (a')} circle (0.8); \node[right] at (-59,18) {$a'$};
\draw[fill=black] (-60,-20){coordinate (b')} circle (0.8); \node[right] at (-59,-18) {$b'$};
\draw[fill=black] (-50,0){coordinate (c')} circle (0.8); \node[right] at (-51,3) {$c'$}; 
\draw[fill=black] (-20,20){coordinate (a)} circle (0.8); \node[left] at (-21,18) {$a$};
\draw[fill=black] (-20,-20){coordinate (b)} circle (0.8); \node[left] at (-21,-18) {$b$};
\draw[fill=black] (-30,0){coordinate (c)} circle (0.8); \node[left] at (-29,3) {$c$};
\draw[fill=black] (0,15){coordinate (x)} {node[right] {$x$}} circle (0.8);
\draw[fill=black] (0,-15){coordinate (y)} {node[right] {$y$}} circle (0.8);
\draw[thick] (a')--(b')--(c')--(a');
\draw[thick] (x)--(y)--(a)--(b)--(c)--(x)--(a)--(c)--(y)--(b)--(x);
\draw[red,dotted,thick] (a)--(a') (b)--(b') (c)--(c'); 
\node[below] at (-30,-22) {$\mathcal{L}_1(x,y)$};
\end{scope}

\begin{scope}[xshift=60cm]
\draw[fill=black] (-40,20){coordinate (a)} circle (0.8); \node[right] at (-41,16) {$a$};
\draw[fill=black] (-40,-20){coordinate (a')} circle (0.8); \node[right] at (-41,-16) {$a'$};
\draw[fill=black] (-20,10){coordinate (b)} circle (0.8); \node[right] at (-20,8) {$b$};
\draw[fill=black] (-20,-10){coordinate (b')} circle (0.8); \node[right] at (-20,-8) {$b'$};
\draw[fill=black] (-0,15){coordinate (x)} {node[right] {$x$}} circle (0.8);
\draw[fill=black] (-0,-15){coordinate (y)} {node[right] {$y$}} circle (0.8);
\draw[thick] (x)--(a)--(b)--(x);
\draw[thick] (y)--(a')--(b')--(y);
\draw[red,dotted,thick] (a)--(a') (b)--(b'); 
\node[below] at (-20,-22) {$\mathcal{L}_2(x,y)$};
\end{scope}

\end{tikzpicture}
\caption{Classes of links $\mathcal{L}_1(x,y)$ and $\mathcal{L}_2(x,y)$.}\label{L12}
\end{figure}
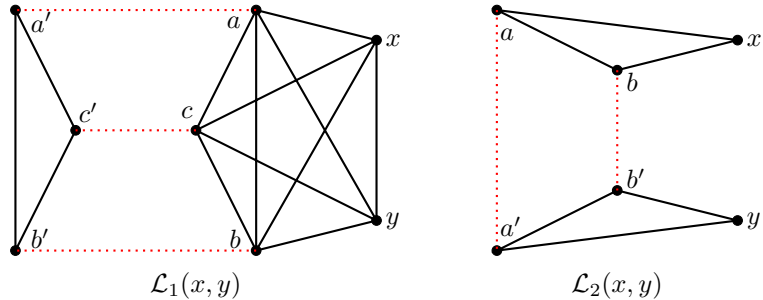

We now define a chain, with two labeled pairs of vertices, as follows: Let $\{x_1,y_1\}$, $\{x_2,y_2\}$ be two pair of vertices, possibly $\{x_1,y_1\}\cap\{x_2,y_2\}\neq\emptyset$. A \emph{bipath chain} from $\{x_1,y_1\}$ to $\{x_2,y_2\}$, is the union of two vertex-disjoint paths from $\{x_1,y_1\}$ to $\{x_2,y_2\}$; and a \emph{triangle chain} from $\{x_1,y_1\}$ to $\{x_2,y_2\}$, is a graph with vertex set $\{a_1,a_2,\ldots,a_k\}$, $k\geq 3$, where $\{x_1,y_1\}=\{a_1,a_2\}$ and $\{x_2,y_2\}=\{a_{k-1},a_k\}$, such that $a_i,a_j$ are adjacent if and only if $j=i+1$ or $j=i+2$. Both bipath chains and triangle chains are called \emph{pure chains}. If a bipath chain consists of two trivial paths, i.e., $\{x_1,y_1\}=\{x_2,y_2\}$, then we shall call it a \emph{trivial chain}. We will denote by $H(\{x_1,y_1\},\{x_2,y_2\})$ the pure chain $H$ from $\{x_1,y_1\}$ to $\{x_2,y_2\}$.

We further define a \emph{chain} $H$ as the graph with labeled pairs of vertices $\{x_1,y_1\}$ and $\{x_2,y_2\}$, that obtained from a series of vertex-disjoint pure chains $H_i=H_i(\{a_i,b_i\},\{c_i,d_i\})$, $i=1,\ldots,k$, where $\{x_1,y_1\}=\{a_1,b_1\}$ and $\{x_2,y_2\}=\{c_k,d_k\}$, by adding all missing edges between vertices in $\{c_i,d_i,a_{i+1},b_{i+1}\}$ for $i=1,\ldots,k-1$ (making $\langle \{c_i,d_i,a_{i+1},$ $b_{i+1}\}\rangle$ to be a 4-clique) (see Figure \ref{FiChain}). The pair $\{x_1,y_1\}$ is the origin of $H$ and $\{x_2,y_2\}$ is the terminus of $H$. We also denote the chain by $H(\{x_1,y_1\},\{x_2,y_2\})$, and call it a chain from $\{x_1,y_1\}$ to $\{x_2,y_2\}$, or connecting $\{x_1,y_1\}$ and $\{x_2,y_2\}$. If both $H_1,H_k$ are bipath chains, then $H$ is of type BB (including the case that $k=1$ and $H=H_1$ is a bipath chain); if $H_1$ is a triangle chain and $H_k$ is a bipath chain, then $H$ is of type TB. The types BT and TT are similar.

\begin{figure}[h]
\begin{center}
\begin{tikzpicture}[scale=0.1]

\begin{scope}
\foreach \x in {1,2,3,4} \draw[fill=black] (\x*10-10,0) {coordinate (a1\x)} circle (0.6);
\foreach \x in {1,2} \draw[fill=black] (\x*10,15) {coordinate (a2\x)} circle (0.6);
\draw[thick] (a11)--(a14) (a21)--(a22);
\draw[dotted, thick] (-2,-2) rectangle (32,17);
\node[below] at (15,-3) {$H_1$};
\end{scope}

\begin{scope}[xshift=40cm]
\foreach \x in {1,2,3} \draw[fill=black] (\x*10-10,0) {coordinate (b1\x)} circle (0.6);
\foreach \x in {1,2,3} \draw[fill=black] (\x*10-10,15) {coordinate (b2\x)} circle (0.6);
\draw[thick] (b11)--(b13) (b21)--(b23);
\draw[thick] (b21)--(b11)--(b22)--(b12)--(b23)--(b13);
\draw[dotted, thick] (-2,-2) rectangle (22,17);
\node[below] at (10,-3) {$H_2$};
\end{scope}

\begin{scope}[xshift=70cm]
\draw[fill=black] (0,0) {coordinate (c11)} circle (0.6);
\draw[fill=black] (0,15) {coordinate (c21)} circle (0.6);
\draw[dotted, thick] (-2,-2) rectangle (2,17);
\node[below] at (0,-3) {$H_3$};
\end{scope}

\begin{scope}[xshift=80cm]
\draw[fill=black] (5,0) {coordinate (d11)} circle (0.6);
\foreach \x in {1,2} \draw[fill=black] (\x*10-10,15) {coordinate (d2\x)} circle (0.6);
\draw[thick] (d21)--(d22);
\draw[dotted, thick] (-2,-2) rectangle (12,17);
\node[below] at (5,-3) {$H_4$};
\end{scope}

\begin{scope}[xshift=100cm]
\foreach \x in {1,2,3} \draw[fill=black] (\x*10-10,0) {coordinate (e1\x)} circle (0.6);
\foreach \x in {1,2} \draw[fill=black] (\x*10-5,15) {coordinate (e2\x)} circle (0.6);
\draw[thick] (e11)--(e13) (e21)--(e22);
\draw[thick] (e11)--(e21)--(e12)--(e22)--(e13);
\draw[dotted, thick] (-2,-2) rectangle (22,17);
\node[below] at (10,-3) {$H_5$};
\end{scope}

\begin{scope}[yshift=-25cm]
\foreach \x in {1,2,3,4} \draw[fill=black] (\x*10-10,0) {coordinate (a1\x)} circle (0.6);
\foreach \x in {1,2} \draw[fill=black] (\x*10,15) {coordinate (a2\x)} circle (0.6);
\draw[thick] (a11)--(a14) (a21)--(a22); 
\foreach \x in {1,2,3} \draw[fill=black] (\x*10+30,0) {coordinate (b1\x)} circle (0.6);
\foreach \x in {1,2,3} \draw[fill=black] (\x*10+30,15) {coordinate (b2\x)} circle (0.6);
\draw[thick] (b11)--(b13) (b21)--(b23);
\draw[thick] (b21)--(b11)--(b22)--(b12)--(b23)--(b13); 
\draw[fill=black] (70,0) {coordinate (c11)} circle (0.6);
\draw[fill=black] (70,15) {coordinate (c21)} circle (0.6); 
\draw[fill=black] (85,0) {coordinate (d11)} circle (0.6);
\foreach \x in {1,2} \draw[fill=black] (\x*10+70,15) {coordinate (d2\x)} circle (0.6);
\draw[thick] (d21)--(d22); 
\foreach \x in {1,2,3} \draw[fill=black] (\x*10+90,0) {coordinate (e1\x)} circle (0.6);
\foreach \x in {1,2} \draw[fill=black] (\x*10+95,15) {coordinate (e2\x)} circle (0.6);
\draw[thick] (e11)--(e13) (e21)--(e22);
\draw[thick] (e11)--(e21)--(e12)--(e22)--(e13); 
\draw[thick] (a14)--(a22) (a14)--(b11)--(a22)--(b21)--(a14);
\draw[thick] (b13)--(c11)--(b23)--(c21)--(b13) (c11)--(c21);
\draw[thick] (c11)--(d11)--(c21)--(d21)--(c11) (d11)--(d21);
\draw[thick] (d11)--(d22) (d11)--(e11)--(d22)--(e21)--(d11);
\node[left] at (a21) {$x_1$}; \node[left] at (a11) {$y_1$};
\node[right] at (e22) {$x_2$}; \node[right] at (e13) {$y_2$};
\node[below] at (60,-2) {$H(\{x_1,y_1\},\{x_2,y_2\})$};
\end{scope}

\end{tikzpicture}
\caption{A chain from $\{x_1,y_1\}$ to $\{x_2,y_2\}$.}\label{FiChain}
\end{center}
\end{figure}
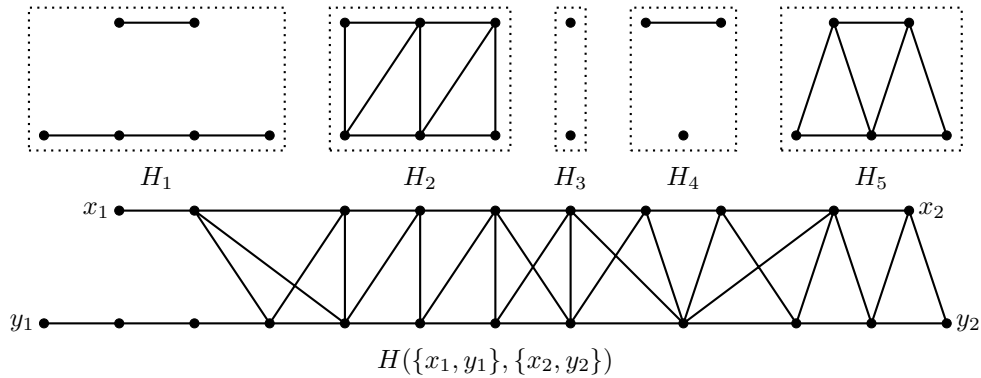

\subsection{Characterization of minimal graphs}

We define seven classes of graphs $\mathcal{H}_i$, $i=1,\ldots,7$, as shown in Figures \ref{H1234} and \ref{H567}. We remark that in Figure \ref{H567}, a red $\boxtimes$ stands for a chain $H$ from $\{x_1,y_1\}$ to $\{x_2,y_2\}$. Let
\begin{itemize}
\item $\mathcal{H}_1$ be the family of graphs obtained from four triangles $\langle \{x_1,y_1,z_1\}\rangle$, $\langle \{x_2,y_2,z_2\}\rangle$, $\langle \{x_3,y_3,z_3\}\rangle$, $\langle \{x'_1,y'_2,z'_3\}\rangle$, by connecting each of the following pairs of vertices, $\{x_1,x'_1\}$, $\{y_2,y'_2\}$, $\{z_3,z'_3\}$, $\{z_1,z_2\}$, $\{x_2,x_3\}$, $\{y_3,y_1\}$, by a pure link;
\item $\mathcal{H}_2$ be the family of graphs obtained from three triangles $\langle \{x_1,y_1,z_1\}\rangle$, $\langle \{x_2,y_2,z_2\}\rangle$, $\langle \{x_3,y_3,z_3\}\rangle$, and a 5-clique $\langle \{x'_1,x'_2,y'_2,x'_3,z'_3\}\rangle$, by connecting each of the following pairs of vertices, $\{x_1,x'_1\}$, $\{x_2,x'_2\}$, $\{y_2,y'_2\}$, $\{x_3,x'_3\}$, $\{z_3,z'_3\}$, $\{z_1,z_2\}$, $\{y_3,y_1\}$, by a pure link;
\item $\mathcal{H}_3$ be the family of graphs obtained from three triangles $\langle \{x_1,y_1,z_1\}\rangle$, $\langle \{x_2,y_2,z_2\}\rangle$, $\langle \{x_3,y_3,z_3\}\rangle$, and a 7-clique $\langle \{x'_1,y'_1,z'_1,y'_2,z'_2,y'_3,z'_3\}\rangle$, by connecting each of the following pairs of vertices, $\{x_1,x'_1\}$, $\{y_1,y'_1\}$, $\{z_1,z'_1\}$, $\{y_2,y'_2\}$, $\{z_2,z'_2\}$, $\{y_3,y'_3\}$, $\{z_3,z'_3\}$, $\{x_2,x_3\}$ by a pure link;
\item $\mathcal{H}_4$ be the family of graphs obtained from three triangles $\langle \{x_1,y_1,z_1\}\rangle$, $\langle \{x_2,y_2,z_2\}\rangle$, $\langle \{x_3,y_3,z_3\}\rangle$, and a 9-clique $\langle \{x'_1,y'_1,z'_1,x'_2,y'_2,z'_2,x'_3,y'_3,z'_3\}\rangle$, by connecting each of the following pairs of vertices, $\{x_i,x'_i\}$, $\{y_i,y'_i\}$, $\{z_i,z'_i\}$, $i=1,2,3$, by a pure link.
\item $\mathcal{H}_5$ be the family of graphs consisting of
two links $L_1=L_1(x_1,y_1)\in
\mathcal{L}_1(x_1,y_1)$,
$L_2=L_2(x_2,y_2)\in
\mathcal{L}_1(x_2,y_2)$
and a chain
$H=H(\{x_1,y_1\},\{x_2,y_2\})$, such that $L_1,L_2$ and $H$ are internally-disjoint;
\item $\mathcal{H}_6$ be the family of graphs 
consisting of
two links $L_1=L_1(x_1,y_1)\in
\mathcal{L}_1(x_1,y_1)$,
$L_2=L_2(x_2,y_2)\in
\mathcal{L}_2(x_2,y_2)$ and a chain $H=H(\{x_1,y_1\},\{x_2,y_2\})$ of type BB or TB,
such that $L_1,L_2$ and $H$ are internally-disjoint;

\item $\mathcal{H}_7$ 
be the family of graphs 
consisting of
two links $L_1=L_1(x_1,y_1)\in
\mathcal{L}_2(x_1,y_1)$,
$L_2=L_2(x_2,y_2)\in
\mathcal{L}_2(x_2,y_2)$ and a chain $H=H(\{x_1,y_1\},\{x_2,y_2\})$ of type BB,
such that $L_1,L_2$ and $H$ are internally-disjoint.
\end{itemize}

\begin{figure}[h]
\centering
\begin{tikzpicture}[scale=0.08]

\begin{scope}
\foreach \x/\y in {1/x,2/y,3/z} 
\draw[fill=black] (210-\x*120:12) {coordinate (\y'\x)} circle (0.8); 
\node at (74:13) {$x'_1$}; \node at (314:13) {$y'_2$}; \node at (226:13) {$z'_3$};
\draw[thick] (x'1)--(y'2)--(z'3)--(x'1); 
\draw[fill=black] (90:27) {coordinate (x1)} circle (0.8); \node at (82:26) {$x_1$};
\draw[fill=black] (110:36) {coordinate (y1)} circle (0.8); \node at (114:37) {$y_1$};
\draw[fill=black] (70:36) {coordinate (z1)} circle (0.8); \node at (66:37) {$z_1$}; 
\draw[thick] (x1)--(y1)--(z1)--(x1); 
\draw[fill=black] (310:36) {coordinate (x2)} circle (0.8); \node at (314:39) {$x_2$};
\draw[fill=black] (330:27) {coordinate (y2)} {node[above] {$y_2$}} circle (0.8);
\draw[fill=black] (350:36) {coordinate (z2)} circle (0.8); \node at (346:39) {$z_2$}; 
\draw[thick] (x2)--(y2)--(z2)--(x2); 
\draw[fill=black] (230:36) {coordinate (x3)} circle (0.8); \node at (226:39) {$x_3$};
\draw[fill=black] (190:36) {coordinate (y3)} circle (0.8); \node at (194:39) {$y_3$};
\draw[fill=black] (210:27) {coordinate (z3)} {node[above] {$z_3$}} circle (0.8); 
\draw[thick] (x3)--(y3)--(z3)--(x3); 
\draw[red, dotted, thick] (x1)--(x'1) (y2)--(y'2) (z3)--(z'3);
\draw[red, dotted, thick] (x2)--(x3) (y3)--(y1) (z1)--(z2);
\node[below] at (0,-30) {$\mathcal{H}_1$};
\end{scope}

\begin{scope}[xshift=90cm]
\foreach \x in {1,2,...,5} 
\draw[fill=black] (162-\x*72:14){coordinate (a\x)} circle (0.8); 
\node at (75:15) {$x'_1$}; 
\node[below] at (a3) {$x'_2$}; \node[below] at (a4) {$x'_3$};
\node[right] at (a2) {$y'_2$}; \node[left] at (a5) {$z'_3$};
\draw[thick] (a1)--(a2)--(a3)--(a4)--(a5)--(a1)--(a3)--(a5)--(a2)--(a4)--(a1); 
\draw[fill=black] (90:27) {coordinate (x1)} circle (0.8); \node at (82:26) {$x_1$};
\draw[fill=black] (110:36) {coordinate (y1)} circle (0.8); \node at (114:37) {$y_1$};
\draw[fill=black] (70:36) {coordinate (z1)} circle (0.8); \node at (66:37) {$z_1$}; 
\draw[thick] (x1)--(y1)--(z1)--(x1); 
\draw[fill=black] (310:36) {coordinate (x2)} circle (0.8); \node at (314:39) {$x_2$};
\draw[fill=black] (330:27) {coordinate (y2)} {node[above] {$y_2$}} circle (0.8);
\draw[fill=black] (350:36) {coordinate (z2)} circle (0.8); \node at (346:39) {$z_2$}; 
\draw[thick] (x2)--(y2)--(z2)--(x2); 
\draw[fill=black] (230:36) {coordinate (x3)} circle (0.8); \node at (226:39) {$x_3$};
\draw[fill=black] (190:36) {coordinate (y3)} circle (0.8); \node at (194:39) {$y_3$};
\draw[fill=black] (210:27) {coordinate (z3)} {node[above] {$z_3$}} circle (0.8); 
\draw[thick] (x3)--(y3)--(z3)--(x3); 
\draw[red, dotted, thick] (x1)--(a1) (y2)--(a2) (x2)--(a3) (x3)--(a4) (z3)--(a5);
\draw[red, dotted, thick] (y3)--(y1) (z1)--(z2);
\node[below] at (0,-30) {$\mathcal{H}_2$};
\end{scope}

\begin{scope}[yshift=-75cm]
\foreach \x in {1,2,...,7} 
{\draw[fill=black] (141.4-\x*51.4:16){coordinate (a\x)} circle (0.8); 
\coordinate (b\x) at (141.4-\x*51.4:12);}
\draw(0,0) {node {$K_7$}} circle (18);
\node at (b1) {$x'_1$}; \node at (b2) {$z'_1$}; \node at (b3) {$z'_2$}; \node at (b4) {$y'_2$};
\node at (b5) {$z'_3$}; \node at (b6) {$y'_3$}; \node at (b7) {$y'_1$};
\draw[fill=black] (90:27) {coordinate (x1)} circle (0.8); \node at (82:26) {$x_1$};
\draw[fill=black] (110:36) {coordinate (y1)} circle (0.8); \node at (114:37) {$y_1$};
\draw[fill=black] (70:36) {coordinate (z1)} circle (0.8); \node at (66:37) {$z_1$}; 
\draw[thick] (x1)--(y1)--(z1)--(x1); 
\draw[fill=black] (310:36) {coordinate (x2)} circle (0.8); \node at (314:39) {$x_2$};
\draw[fill=black] (330:27) {coordinate (y2)} {node[above] {$y_2$}} circle (0.8);
\draw[fill=black] (350:36) {coordinate (z2)} circle (0.8); \node at (346:39) {$z_2$}; 
\draw[thick] (x2)--(y2)--(z2)--(x2); 
\draw[fill=black] (230:36) {coordinate (x3)} circle (0.8); \node at (226:39) {$x_3$};
\draw[fill=black] (190:36) {coordinate (y3)} circle (0.8); \node at (194:39) {$y_3$};
\draw[fill=black] (210:27) {coordinate (z3)} {node[above] {$z_3$}} circle (0.8); 
\draw[thick] (x3)--(y3)--(z3)--(x3); 
\draw[red, dotted, thick] (x1)--(a1) (z1)--(a2) (z2)--(a3) (y2)--(a4) (z3)--(a5) (y3)--(a6) (y1)--(a7);
\draw[red, dotted, thick] (x2)--(x3);
\node[below] at (0,-30) {$\mathcal{H}_3$};
\end{scope}

\begin{scope}[xshift=90cm,yshift=-75cm]
\foreach \x in {1,2,...,9} 
{\draw[fill=black] (130-\x*40:18){coordinate (a\x)} circle (0.8); 
\coordinate (b\x) at (130-\x*40:14);}
\draw(0,0) {node {$K_9$}} circle (20);
\node at (b1) {$x'_1$}; \node at (b2) {$z'_1$}; \node at (b3) {$z'_2$}; 
\node at (b4) {$y'_2$}; \node at (b5) {$x'_2$}; \node at (b6) {$x'_3$}; 
\node at (b7) {$z'_3$}; \node at (b8) {$y'_3$}; \node at (b9) {$y'_1$};
\draw[fill=black] (90:27) {coordinate (x1)} circle (0.8); \node at (82:26) {$x_1$};
\draw[fill=black] (110:36) {coordinate (y1)} circle (0.8); \node at (114:37) {$y_1$};
\draw[fill=black] (70:36) {coordinate (z1)} circle (0.8); \node at (66:37) {$z_1$}; 
\draw[thick] (x1)--(y1)--(z1)--(x1); 
\draw[fill=black] (310:36) {coordinate (x2)} circle (0.8); \node at (314:39) {$x_2$};
\draw[fill=black] (330:27) {coordinate (y2)} {node[above] {$y_2$}} circle (0.8);
\draw[fill=black] (350:36) {coordinate (z2)} circle (0.8); \node at (346:39) {$z_2$}; 
\draw[thick] (x2)--(y2)--(z2)--(x2); 
\draw[fill=black] (230:36) {coordinate (x3)} circle (0.8); \node at (226:39) {$x_3$};
\draw[fill=black] (190:36) {coordinate (y3)} circle (0.8); \node at (194:39) {$y_3$};
\draw[fill=black] (210:27) {coordinate (z3)} {node[above] {$z_3$}} circle (0.8); 
\draw[thick] (x3)--(y3)--(z3)--(x3); 
\draw[red, dotted, thick] (x1)--(a1) (z1)--(a2) (z2)--(a3) (y2)--(a4) (x2)--(a5); 
\draw[red, dotted, thick] (x3)--(a6) (z3)--(a7) (y3)--(a8) (y1)--(a9);
\node[below] at (0,-30) {$\mathcal{H}_4$};
\end{scope}

\end{tikzpicture}
\caption{Classes of graphs $\mathcal{H}_1, \mathcal{H}_2, \mathcal{H}_3, \mathcal{H}_4$.}\label{H1234}
\end{figure}
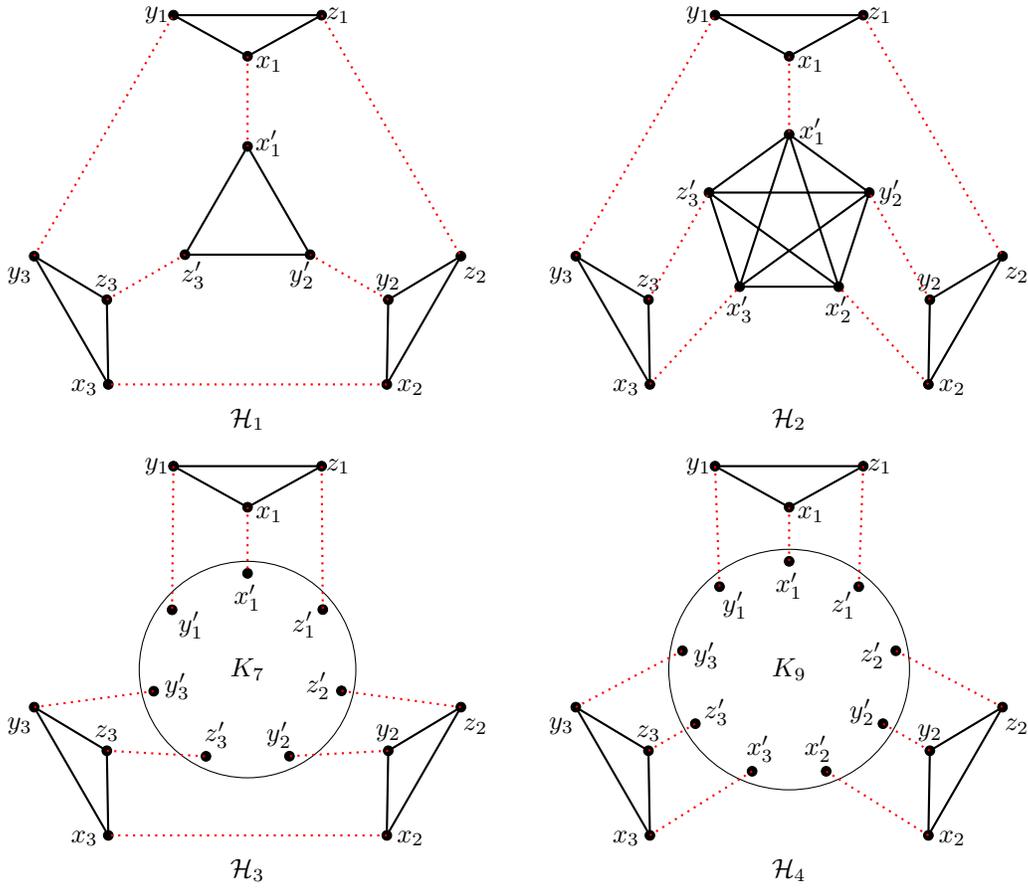

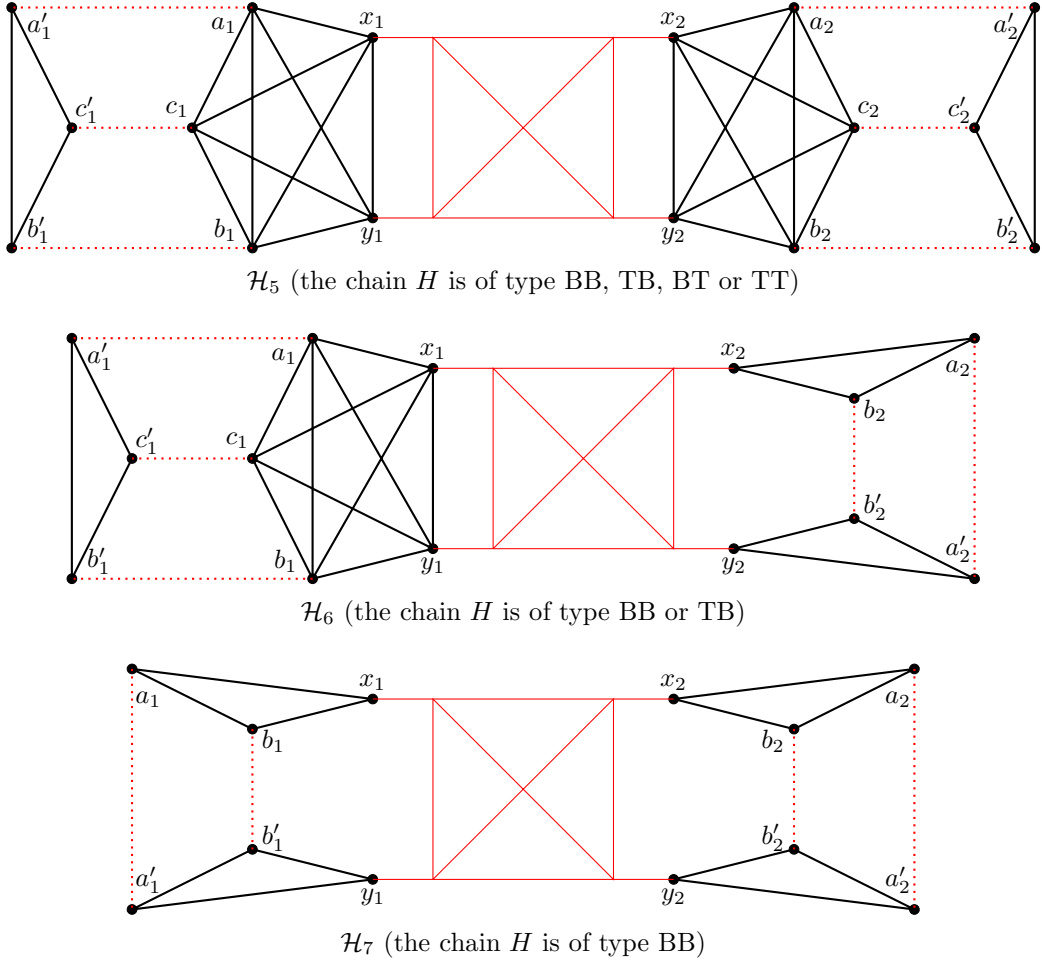
\begin{figure}[h]
\centering
\begin{tikzpicture}[scale=0.08]

\begin{scope}
\draw[fill=black] (-85,20){coordinate (a'1)} circle (0.8); \node[right] at (-84,17) {$a'_1$};
\draw[fill=black] (-85,-20){coordinate (b'1)} circle (0.8); \node[right] at (-84,-17) {$b'_1$};
\draw[fill=black] (-75,0){coordinate (c'1)} circle (0.8); \node[right] at (-76,3) {$c'_1$}; 
\draw[fill=black] (-45,20){coordinate (a1)} circle (0.8); \node[left] at (-46,17) {$a_1$};
\draw[fill=black] (-45,-20){coordinate (b1)} circle (0.8); \node[left] at (-46,-17) {$b_1$};
\draw[fill=black] (-55,0){coordinate (c1)} circle (0.8); \node[left] at (-54,3) {$c_1$};
\draw[fill=black] (-25,15){coordinate (x1)} {node[above] {$x_1$}} circle (0.8);
\draw[fill=black] (-25,-15){coordinate (y1)} {node[below] {$y_1$}} circle (0.8);
\draw[thick] (a'1)--(b'1)--(c'1)--(a'1);
\draw[thick] (x1)--(y1)--(a1)--(b1)--(c1)--(x1)--(a1)--(c1)--(y1)--(b1)--(x1);
\draw[red,dotted,thick] (a1)--(a'1) (b1)--(b'1) (c1)--(c'1); 
\draw[fill=black] (85,20){coordinate (a'2)} circle (0.8); \node[left] at (84,17) {$a'_2$};
\draw[fill=black] (85,-20){coordinate (b'2)} circle (0.8); \node[left] at (84,-17) {$b'_2$};
\draw[fill=black] (75,0){coordinate (c'2)} circle (0.8); \node[left] at (76,3) {$c'_2$};
\draw[fill=black] (45,20){coordinate (a2)} circle (0.8); \node[right] at (46,17) {$a_2$};
\draw[fill=black] (45,-20){coordinate (b2)} circle (0.8); \node[right] at (46,-17) {$b_2$};
\draw[fill=black] (55,0){coordinate (c2)} circle (0.8); \node[right] at (54,3) {$c_2$};
\draw[fill=black] (25,15){coordinate (x2)} {node[above] {$x_2$}} circle (0.8);
\draw[fill=black] (25,-15){coordinate (y2)} {node[below] {$y_2$}} circle (0.8);
\draw[thick] (a'2)--(b'2)--(c'2)--(a'2);
\draw[thick] (x2)--(y2)--(a2)--(b2)--(c2)--(x2)--(a2)--(c2)--(y2)--(b2)--(x2);
\draw[red,dotted,thick] (a2)--(a'2) (b2)--(b'2) (c2)--(c'2); 
\draw[red] (x1)--(x2) (y1)--(y2);
\draw[red] (-15,-15)--(-15,15)--(15,-15)--(15,15)--(-15,-15); 
\node[below] at (0,-22) {$\mathcal{H}_5$ (the chain $H$ is of type BB, TB, BT or TT)};
\end{scope}

\begin{scope}[xshift=10cm,yshift=-55cm]
\draw[fill=black] (-85,20){coordinate (a'1)} circle (0.8); \node[right] at (-84,17) {$a'_1$};
\draw[fill=black] (-85,-20){coordinate (b'1)} circle (0.8); \node[right] at (-84,-17) {$b'_1$};
\draw[fill=black] (-75,0){coordinate (c'1)} circle (0.8); \node[right] at (-76,3) {$c'_1$}; 
\draw[fill=black] (-45,20){coordinate (a1)} circle (0.8); \node[left] at (-46,17) {$a_1$};
\draw[fill=black] (-45,-20){coordinate (b1)} circle (0.8); \node[left] at (-46,-17) {$b_1$};
\draw[fill=black] (-55,0){coordinate (c1)} circle (0.8); \node[left] at (-54,3) {$c_1$};
\draw[fill=black] (-25,15){coordinate (x1)} {node[above] {$x_1$}} circle (0.8);
\draw[fill=black] (-25,-15){coordinate (y1)} {node[below] {$y_1$}} circle (0.8);
\draw[thick] (a'1)--(b'1)--(c'1)--(a'1);
\draw[thick] (x1)--(y1)--(a1)--(b1)--(c1)--(x1)--(a1)--(c1)--(y1)--(b1)--(x1);
\draw[red,dotted,thick] (a1)--(a'1) (b1)--(b'1) (c1)--(c'1); 
\draw[fill=black] (65,20){coordinate (a2)} circle (0.8); \node[left] at (66,15) {$a_2$};
\draw[fill=black] (65,-20){coordinate (a'2)} circle (0.8); \node[left] at (66,-15) {$a'_2$};
\draw[fill=black] (45,10){coordinate (b2)} circle (0.8); \node[right] at (45,8) {$b_2$};
\draw[fill=black] (45,-10){coordinate (b'2)} circle (0.8); \node[right] at (45,-8) {$b'_2$};
\draw[fill=black] (25,15){coordinate (x2)} {node[above] {$x_2$}} circle (0.8);
\draw[fill=black] (25,-15){coordinate (y2)} {node[below] {$y_2$}} circle (0.8);
\draw[thick] (x2)--(a2)--(b2)--(x2);
\draw[thick] (y2)--(a'2)--(b'2)--(y2);
\draw[red,dotted,thick] (a2)--(a'2) (b2)--(b'2); 
\draw[red] (x1)--(x2) (y1)--(y2);
\draw[red] (-15,-15)--(-15,15)--(15,-15)--(15,15)--(-15,-15); 
\node[below] at (-10,-22) {$\mathcal{H}_6$ (the chain $H$ is of type BB or TB)};
\end{scope}

\begin{scope}[yshift=-110cm]
\draw[fill=black] (-65,20){coordinate (a1)} circle (0.8); \node[right] at (-66,15) {$a_1$};
\draw[fill=black] (-65,-20){coordinate (a'1)} circle (0.8); \node[right] at (-66,-15) {$a'_1$};
\draw[fill=black] (-45,10){coordinate (b1)} circle (0.8); \node[right] at (-45,8) {$b_1$};
\draw[fill=black] (-45,-10){coordinate (b'1)} circle (0.8); \node[right] at (-45,-8) {$b'_1$};
\draw[fill=black] (-25,15){coordinate (x1)} {node[above] {$x_1$}} circle (0.8);
\draw[fill=black] (-25,-15){coordinate (y1)} {node[below] {$y_1$}} circle (0.8);
\draw[thick] (x1)--(a1)--(b1)--(x1);
\draw[thick] (y1)--(a'1)--(b'1)--(y1);
\draw[red,dotted,thick] (a1)--(a'1) (b1)--(b'1); 
\draw[fill=black] (65,20){coordinate (a2)} circle (0.8); \node[left] at (66,15) {$a_2$};
\draw[fill=black] (65,-20){coordinate (a'2)} circle (0.8); \node[left] at (66,-15) {$a'_2$};
\draw[fill=black] (45,10){coordinate (b2)} circle (0.8); \node[left] at (45,8) {$b_2$};
\draw[fill=black] (45,-10){coordinate (b'2)} circle (0.8); \node[left] at (45,-8) {$b'_2$};
\draw[fill=black] (25,15){coordinate (x2)} {node[above] {$x_2$}} circle (0.8);
\draw[fill=black] (25,-15){coordinate (y2)} {node[below] {$y_2$}} circle (0.8);
\draw[thick] (x2)--(a2)--(b2)--(x2);
\draw[thick] (y2)--(a'2)--(b'2)--(y2);
\draw[red,dotted,thick] (a2)--(a'2) (b2)--(b'2); 
\draw[red] (x1)--(x2) (y1)--(y2);
\draw[red] (-15,-15)--(-15,15)--(15,-15)--(15,15)--(-15,-15); 
\node[below] at (0,-22) {$\mathcal{H}_7$ (the chain $H$ is of type BB)};
\end{scope}

\end{tikzpicture}
\caption{Classes $\mathcal{H}_5$, $\mathcal{H}_6$ and $\mathcal{H}_7$.}\label{H567}
\end{figure}

We are now in a position to state our second main result.

\begin{theorem}\label{Th3}
Let $G$ be a $2$-connected claw-free non-cycle graph without spanning $\varTheta$-subgraphs. If any proper induced subgraph of $G$ either has connectivity less than $2$, or is a cycle, or has a spanning $\varTheta$-subgraph, then $G\in\mathcal{H}_1\cup\mathcal{H}_2\cup\mathcal{H}_3\cup\mathcal{H}_4\cup\mathcal{H}_5\cup\mathcal{H}_6\cup\mathcal{H}_7$.
\end{theorem}

\section{Preliminaries}

\subsection{Some useful results}

We shall make use of the following theorems. Denote by $\alpha(G), \kappa(G)$ the independence number and the connectivity of a graph $G$, respectively. A graph $G$ is said to be \emph{locally connected} if $\langle N(v)\rangle$ is connected for all $v\in V(G)$.

\begin{theorem}[Oberly, Simic and Sumner \cite{6}]\label{Th A}
 If $G$ is a connected, locally connected claw-free graph of order at least $3$, then $G$ is hamiltonian.
\end{theorem}

\begin{theorem}[Egawa~\cite{4}]\label{Th B}
Let $G$ be a connected non-complete $P_4$-free graph, and let $S$ be a minimum vertex-cut of $G$. Then each vertex in $S$ is adjacent to all vertices in $G-S$.
\end{theorem}

\begin{theorem}[Chv\'atal and Erd\H{o}s \cite{3}]\label{Th D}
Let $G$ be a $2$-connected graph. If $\alpha(G)\leq\kappa(G)$, then $G$ is hamiltonian.
\end{theorem}

Let $G$ be a graph.
By a \emph{$2$-cut} we mean a vertex-cut of $G$ with exactly 2 vertices. Let $\{x,y\}$ be a 2-cut of $G$ and $H$ be a component of $G-\{x,y\}$. We call the induced subgraph $L=\langle\{x,y\}\cup V(H)\rangle$ a \emph{link of $G$} connecting $x,y$. We shall use $L(x,y)$ to denote such a link of $G$. If $G-\{x,y\}$ has exactly two components, then $G-H$ is also a link of $G$ connecting $x,y$, which is called the \emph{co-link} of the link $L$. We shall denote by $L^*$ (or $L^*(x,y)$) the co-link of $L$.

\begin{lemma}\label{LeCutxy}
Let $G$ be a $2$-connected graph, $\{x_1,x_2\}$ be a $2$-cut of $G$ and $H$ be a component of $G-\{x_1,x_2\}$. If $G$ contains a spanning $\varTheta$-subgraph $T$, then each $x_i$ has a neighbor $z_i$
in $H$ such that $x_iz_i\in E(T)$ for $i=1,2$.
\end{lemma}

\begin{proof}
This can be deduced by the fact that $T$ is a 2-connected spanning subgraph of $G$.
\end{proof}

\subsection{The unfoldments of loopless multi-graphs}

By a \emph{multi-graph} we mean a graph in which multi-edges or loops are allowed, and a multi-graph is \emph{loopless} if it does not contain loops. A coloring of a graph here stands for a vertex-coloring with two colors red and blue, and a \emph{colored graph} means one with such a vertex-coloring. We will always assume that the concerning graphs has no isolated vertices in the following. 

Let $F$ be a loopless multi-graph. We define a colored graph $G$, called an \emph{unfoldment} of $F$, as follows: First for every edge $e\in E(F)$ with end-vertices $u,v$, we assign two distinct vertices $u^e,v^e$ and a pure link $L^e$ connecting $u^e$ and $v^e$. Second for every vertex $u\in V(F)$, we assign a clique $K^u$ with vertex set $\{u^e: e\in E(F)\mbox{ and }u\mbox{ is incident with }e\}$ and we remark that $d_F(u)=|K^u|$.
Thirdly we assign the inner vertices of all pure links red and the end-vertices of all pure links blue (we remark that for distinct edges $e_1,e_2$ in $F$, the pure links $L^{e_1}$, $L^{e_2}$ are vertex-disjoint). Finally let $G$ be the union of the above pure links and cliques, i.e., $G=(\bigcup_{e\in E(F)}L^e)\cup(\bigcup_{u\in V(F)}K^u)$. 

We remark that $F$ has infinitely many of unfoldments. Notice that all graphs $P_{k_1,k_2,k_3}$, $k_i\geq 2$, are unfoldments of the multi-graph $F$ with two vertices and three edges.

Let $G$ be a colored graph. A path of length at least 2 with both end-vertices blue and all inner-vertices red is called a \emph{feasible} path. Two blue vertices are \emph{associated} if they are connected by a feasible path.

\begin{lemma}\label{LeUnfold}
Let $F$ be a loopless multi-graph and let $G$ be an unfoldment of $F$. Then\\
\indent (1) $G$ is claw-free;\\
\indent (2) all red vertices are of degree $2$;\\
\indent (3) every blue vertex has a unique red neighbor;\\
\indent (4) two associated blue vertices have no common blue neighbors;\\
\indent (5) if two associated blue vertices are adjacent, then they have a common red neighbor.
\\
Consequently, if a colored graph $G$ satisfies (1)-(5), then $G$ is an unfoldment of some loopless multi-graph.
\end{lemma}

\begin{proof}
Note that a red vertex of $G$ is an inner-vertex of some pure link $L^e$. So it has degree 2, and this proves (2). For a blue vertex $x$ of $G$, there is a unique vertex $u\in V(F)$ and a unique edge $e\in E(F)$ incident with $u$ such that $x=u^e$. The only red neighbor of $u^e$ is the inner-vertex of $L^e$ adjacent to $u^e$. Thus (3) holds. 
Let $x,y$ be two associated blue vertcies of $G$.
By the definiton of $G$,
there exists an edge $e=uv\in E(F)$
such that $x=u^e,y=v^e$.
 Then $u^e,v^e$ are end-vertices of $L^e$. Notice that the neighbors of $u^e$ outside $L^e$ are those in $K^u$, and the neighbors of $v^e$ outside $L^e$ are those in $K^v$. Since $V(K^u)\cap V(K^v)=\emptyset$, we see that $u^e,v^e$ have no common blue neighbors, and (4) holds. Since $F$ is loopless, $v^e\notin V(K^u)$. If $u^ev^e\in E(G)$, then it is an edge of $L^e$, implying that $L^e$ is a triangle, and (5) holds. Let $\bar{u}^e$ be the unique red neighbor of $u^e$. If $u^ev^e\in E(G)$, then $N_G(u^e)=\{\bar{u}^e,v^e\}\cup V(K^u)\backslash\{u^e\}$; if $u^ev^e\notin E(G)$, then $N_G(u^e)=\{\bar{u}^e\}\cup V(K^u)\backslash\{u^e\}$. In both case $\langle N_G(u^e)\rangle$ is the disjoint union of two cliques, which implies (1).

Now we suppose that $G$ is a colored graph satisfying (1)-(5). Let $V_b,V_r$ be the sets of blue vertices and red vertices of $G$, respectively. By (2) and (3), we see that every blue vertex is an end-vertex of exactly one feasible path, and that any two feasible paths are vertex-disjoint. It follows that $V_b$ can be partitioned into associated pairs. 

Let $x,y\in V_b$ be two associated vertices, and let $P$ be the feasible path from $x$ to $y$ in $G$. We first show that $\langle V(P)\rangle$ is a pure link from $x$ to $y$. Recall that all vertices in $P-\{x,y\}$ are red and then have degree 2 by (2). We are done if $xy\notin E(G)$. Suppose now that $xy\in E(G)$. By (5) $x,y$ have a common neighbor $z\in V_r$. Since $d(z)=2$, we have that $P=xzy$, and $\langle V(P)\rangle$ is a triangle, which is also a pure link from $x$ to $y$. We shall call $\langle V(P)\rangle$ the \emph{feasible link} of $G$ from $x$ to $y$.

Let $G_b$ be the graph with the vertex set $V_b$ and
the edge set $E(G_b)=\{xy\in E(G): x,y$ are not associated$\}$.
We now claim that each component of $G_b$ is a clique. Suppose otherwise that $G_b$ contains an induced $P_3$, say $xyz$. Let $\bar{y}$ be the red neighbor of $y$ in $G$. If $x\bar{y}\in E(G)$, then $x,y$ are associated and $xy\notin E(G_b)$, a contradiction. Thus we conclude that $x\bar{y}\notin E(G)$, and similarly, $z\bar{y}\notin E(G)$. It follows that $xz\in E(G)$ by (1). Since $xz\notin E(G_b)$, it follows that $x$ and $z$ are associated, but then they have a common blue neighbor $y$, contradicting (4). So as we claimed, $G_b$ is the union of vertex-disjoint cliques. 

Notice that two associated vertices are nonadjacent in $G_b$, implying that they are in distinct components of $G_b$. Now let $F$ be the graph obtained from $G$ by replacing each feasible link by an edge between the end-vertices of the link, and then contracting all edges of $G_b$ (making each component of $G_b$ to be a single vertex). We can see that $G$ is an unfoldment of $F$.
\end{proof}

Conversely, $F$ can also be obtained from its unfoldment $G$ by the following operations: (1) removing all possible edges between associated pairs; (2) contracting all edges between blue vertices; (3) replacing each feasible path by an edge between the end-vertices of the path. It is trivial that $F$ is connected if and only if $G$ is connected. We now prove the following lemma concerning the properties of multi-graphs and their unfoldments.

\begin{lemma}\label{LePropertyFG}
Let $F$ be a loopless multi-graph and let $G$ be an unfoldment of $F$. Suppose $G$ is not a cycle. Then\\
\indent (1) $F$ is $2$-edge-connected if and only if $G$ is $2$-connected;\\
\indent (2) $F$ is $3$-edge-connected if and only if $G$ is $2$-connected and every $2$-cut $\{x,y\}$ of $G$ with both $x,y$ blue is an associated pair;\\ 
\indent (3) $G$ has a Hamilton cycle if and only if $F$ has an Euler tour;\\
\indent (4) $G$ has a spanning $\varTheta$-subgraph if and only if $F$ is $2$-edge-connected and has an Euler trail.
\end{lemma}

\begin{proof}
Recall that every edge $e$ of $F$ corresponds to a unique link $L^e$ of $G$, and that two distinct edges in $F$ correspond to two vertex-disjoint links in $G$. We denote by $P^e$ the feasible path of $G$ contained in $L^e$. That is, if $L^e$ is a path, then $P^e=L^e$; and if $L^e$ is a triangle, then $P^e$ is obtained from $L^e$ by removing the edge between the two end-vertices of $L^e$.

Let $T=u_0e_1u_1e_2u_2\ldots u_{k-1}e_ku_k$ be a trail of $F$. Then $u_0^{e_1}P^{e_1}u_1^{e_1}u_1^{e_2}P^{e_2}u_2^{e_2}\ldots u_{k-1}^{e_k}P^{e_k}u_k^{e_k}$ is a path of $G$, and we denote by $P^T$. Moreover, if $T$ is a closed trail, i.e., $u_0=u_k$, then the above path can be extended to be a cycle by adding the edge $u_0^{e_k}u_0^{e_1}$, which we denote by $C^T$.

(1) By the assumption that $G$ is not a cycle, we see that $G$ has at least 4 blue vertices, and thus $F$ has at least 2 edges. Now $F$ is 2-edge-connected if and only if each two edges of $F$ are contained in a closed trail, if and only if each two vertices of $G$ are contained in a cycle, and if and only if $G$ is 2-connected.

(2) By (1) we can assume that $F$ is 2-edge-connected and $G$ is 2-connected. For the sufficiency,
suppose otherwise that $F$ has an edge-cut $\{e_1,e_2\}$. Let $u_i,v_i$ be the two end-vertices of $e_i$, $i=1,2$, such that $u_1,u_2$ are in a common component of $G-\{e_1,e_2\}$ and $v_1,v_2$ are in another component. Consider the four blue vertices $u_1^{e_1}, v_1^{e_1}, u_2^{e_2}, v_2^{e_2}$ of $G$. Notice that every path of $G$ from $u_1^{e_1}$ to $v_2^{e_2}$ must pass through $v_1^{e_1}$ or $u_2^{e_2}$. It follows that $\{v_1^{e_1},u_2^{e_2}\}$ is a 2-cut of $G$. Moreover the two components of $G-\{v_1^{e_1},u_2^{e_2}\}$ contain $u_1^{e_1}$ and $v_2^{e_2}$, respectively. It follows that $v_1^{e_1}$ and $u_2^{e_2}$ are not associated, a contradiction.

To prove the necessity, 
suppose by contradiction that $G$ has a 2-cut $\{x,y\}$ with both $x,y$ blue and $x,y$ are not associated. Let $z,w$ be the vertices of $G$ associated with $x,y$, respectively. If $z,w$ are in a same component of $G-\{x,y\}$, then $\{x,w\}$ is also a 2-cut of $G$ and $y,z$ are in distinct components of $G-\{x,w\}$. So we can assume without loss of generality that $z,w$ are in distinct components of $G-\{x,y\}$. Let $e_1,e_2\in E(F)$ with end-vertices $u_1,v_1$ and $u_2,v_2$, respectively, such that $x=u_1^{e_1}$, $y=v_2^{e_2}$ (and then $z=v_1^{e_1}$, $w=u_2^{e_2}$). If $F$ contains a path from $v_1$ to $u_2$ avoiding $e_1,e_2$, then $G$ has a path from $z$ to $w$ avoiding $x,y$, contradicting the fact that $\{x,y\}$ is a 2-cut of $G$. This implies that $\{e_1,e_2\}$ is an edge-cut of $F$, 
contradicting the fact that $F$ is 3-edge-connected.

(3) Suppose first that $F$ has an Euler tour $T=u_0e_1u_1e_2u_2\ldots u_{k-1}e_ku_k$, where $u_k=u_0$. Then $$C^T=u_0^{e_1}P^{e_1}u_1^{e_1}u_1^{e_2}P^{e_2}u_2^{e_2}\ldots u_{k-1}^{e_k}P^{e_k}u_k^{e_k}u_0^{e_1}$$ is a Hamilton cycle of $G$.

Conversely, suppose that $G$ has a Hamilton cycle $C$. Notice that every red vertices of $G$ has degree 2. It follows that $C$ contains all feasible paths of $G$. Since each blue vertex is an end-vertex of a feasible path, we can assume that $C=x_1P_1y_1x_2P_2y_2\ldots x_kP_ky_kx_1$, where $P_i$ is a feasible path connecting $x_i$ and $y_i$. For each $P_i$, there is an edge $e_i\in E(F)$ such that $P_i=P^{e_i}$. Since $y_ix_{i+1}$ is an edge of $G$ that is not contained in the link $L^e$ for any $e\in E(F)$. It follows that $y_i$ and $x_{i+1}$ are contained in a clique $K^{v_i}$ for some vertex $v_i\in V(F)$, $i=1,\ldots,k-1$. Moreover, $y_k$ and $x_1$ are contained in a clique $K^{v_0}$ for some vertex $v_0\in V(F)$. Now $T=v_0e_1v_1e_2v_2\ldots v_{k-1}e_kv_0$ is an Euler tour of $F$.

(4) We first prove the sufficiency.
Suppose that $F$ is 2-edge-connected and has an Euler trail.
Then $F$ has at most 2 odd degree vertices.
If $F$ has an Euler tour, then $G$ has a Hamilton cycle by (3). Since $G$ is not a cycle, $G$ has a spanning $\varTheta$-subgraph. Now we assume that $F$ has no Euler tours. It follows that $F$ has exactly 2 vertices of odd degree, say $u_0,v_0$. 

We claim that $F$ has three edge-disjoint paths from $u_0$ to $v_0$. Suppose otherwise that $F$ has an edge-cut separating $u_0,v_0$ with at most two edges. Since $F$ is 2-edge-connected, we can assume that the edge-cut contains exactly two edges $e_1,e_2$. Let $H$ be the component of $F-\{e_1,e_2\}$ containing $u_0$ but not $v_0$. If $e_1,e_2$ are incident with a common vertex in $H$, then $H$ has only one vertex $u_0$ of odd degree, a contradiction. Assume now $e_1,e_2$ are incident with two distinct vertices in $H$, say $u_1,u_2$. If $u_0\in\{u_1,u_2\}$, then $H$ has only one vertex of odd degree (the vertex in $\{u_1,u_2\}$ other than $u_0$); and if $u_0\notin\{u_1,u_2\}$, then $H$ has exactly three vertices $u_0,u_1,u_2$ of odd degree, both contradictions. 

So we conclude that $F$ has three edge-disjoint paths from $u_0$ to $v_0$. Since all vertices of $F$ other than $u_0,v_0$ have even degrees, the three edge-disjoint paths can be extended to be three edge-disjoint trails $T_1,T_2,T_3$ from $u_0$ to $v_0$ covering all edges of $F$.  Let $e_i,f_i$ be the first and last edges of $T_i$, $i=1,2,3$. It follows that $u_0$ is incident with $e_1,e_2,e_3$, and $v_0$ is incident with $f_1,f_2,f_3$. Now $P^{T_1}$, $P^{T_2}$, $P^{T_3}$ are three vertex-disjoint paths of $G$, where $P^{T_i}$ is from $u_0^{e_i}$ to $v_0^{f_i}$, $i=1,2,3$. Since $T_1,T_2,T_3$ cover all edges of $F$, we have that $P^{T_1}$, $P^{T_2}$, $P^{T_3}$ contain all vertices of $G$. Notice that $u_0^{e_1},u_0^{e_2},u_0^{e_3}$ are contained in $K^{u_0}$ and $v_0^{f_1},v_0^{f_2},v_0^{f_3}$ are contained in $K^{v_0}$. The subgraph of $G$ obtained from $P^{T_1}\cup P^{T_2}\cup P^{T_3}$ by adding four edges $u_0^{e_1}u_0^{e_2}$, $u_0^{e_2}u_0^{e_3}$, $v_0^{f_1}v_0^{f_2}$, $v_0^{f_2}v_0^{f_3}$, is a spanning $\varTheta$-subgraph of $G$.

We now prove the necessity.
Suppose now that $G$ has a spanning $\varTheta$-subgraph. Then $G$ is 2-connected, implying that $F$ is 2-edge-connected by (1). Moreover, $G$ contains a Hamilton path. We here show that $G$ has a Hamilton path connecting two blue vertices. Let $P$ be a Hamilton path of $G$ from $x$ to $y$. If both $x,y$ are blue, then we are done. Now suppose that $x$ is a red vertex, and then is of degree 2. 

Let $Q$ be the feasible path of $G$ containing $x$, and let $x_1,x_2$ be the two end-vertices of $Q$. Since all inner vertex of $Q$ is of degree 2, we see that $P$ contains either $Q[x,x_1]$ or $Q[x,x_2]$. We assume without loss of generality that $P$ contains $Q[x,x_2]$, which is, $P[x,x_2]=Q[x,x_2]$. If $|Q[x,x_1]|\geq 2$, then the neighbor of $x$ in $Q[x,x_1]$ must be $y$ (since it has degree 2). Now $P'=P[x_2^+,y]yxP[x,x_2]$, where $x_2^+$ is a successor of $x_2$ on $P$, is also a Hamilton path of $G$ connecting two blue vertices $x_2,x_2^+$, as desired. 

So we assume that $|Q[x,x_1]|=1$ and thus $xx_1\in E(G)$. Notice that $x_1$ appears on $P$ after $x_2$. Let $x_1^-$ be the predecessor of $x_1$ on $P$. Since $x$ is the unique red neighbor of $x_1$, we see that $x_1^-$ is a blue vertex. Now $P'=P[x_1^-,x]xx_1P[x_1,y]$ is a Hamilton path of $G$ with blue origin $x_1^-$. By a symmetric analysis we can find a Hamilton path $P''$ of $G$ with both origin and terminus blue.

Now suppose that $P$ is a Hamilton path of $G$ connecting two blue vertices. Then $P$ contains all feasible paths of $G$, and we can assume that $P=x_1P_1y_1x_2P_2y_2\ldots x_kP_ky_k$, where $P_i$ is a feasible path connecting $x_i$ and $y_i$. By a similar analysis as in (3), we can prove that $F$ has an Euler trail.
\end{proof}

\begin{lemma}\label{LeSubgraphF}
Let $F$ be a loopless multi-graph and let $G$ be an unfoldment of $F$. If $F'$ is a proper subgraph of $F$, then $G$ has a proper induced subgraph $G'$ which is an unfoldment of $F'$.
\end{lemma}

\begin{proof}
The subgraph $G'$ is induced by all the vertices of all the links $L^e$ for $e\in E(F')$.
\end{proof}

\begin{lemma}\label{LeMinimalF}
Let $F$ be a $3$-edge-connected loopless multi-graph with at least four vertices of odd degree. If every $2$-edge-connected proper subgraph of $F$ has at most two vertices of odd degree, then $F$ is isomorphic to $M_1,M_2,M_3$ or $M_4$, as in Figure \ref{M1234}.
\end{lemma}

Note that the unfoldments of $M_i$ are the graphs in $\mathcal{H}_i$, $i=1,2,3,4$.

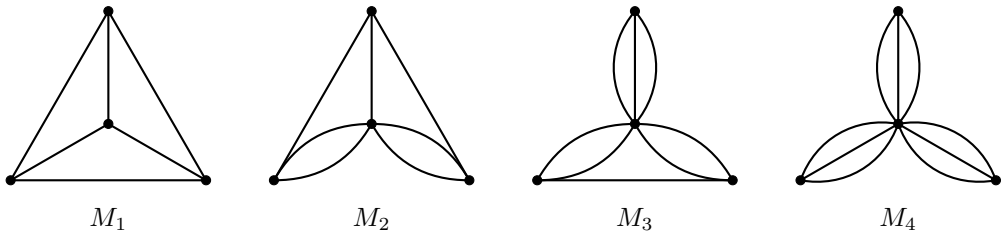
\begin{figure}[h]
\centering
\begin{tikzpicture}[scale=0.1]

\begin{scope}
\draw[fill=black] (0,0) {coordinate (x0)} circle (0.6);
\foreach \x in {1,2,3}
{\draw[fill=black] (\x*120-30:15) {coordinate (x\x)} circle (0.6);
\draw[thick] (x0)--(x\x);}
\draw[thick] (x1)--(x2)--(x3)--(x1);
\node[below] at (0,-10) {$M_1$};
\end{scope}

\begin{scope}[xshift=35cm]
\draw[fill=black] (0,0) {coordinate (x0)} circle (0.6);
\foreach \x in {1,2,3}
{\draw[fill=black] (\x*120-30:15) {coordinate (x\x)} circle (0.6);}
\draw[thick] (x0)--(x1) (x2)--(x1)--(x3);
\foreach \x in {2,3}
{\draw[thick] (x0) to [bend left=30] (x\x);
\draw[thick] (x0) to [bend right=30] (x\x);}
\node[below] at (0,-10) {$M_2$};
\end{scope}

\begin{scope}[xshift=70cm]
\draw[fill=black] (0,0) {coordinate (x0)} circle (0.6);
\foreach \x in {1,2,3}
{\draw[fill=black] (\x*120-30:15) {coordinate (x\x)} circle (0.6);}
\draw[thick] (x0)--(x1) (x2)--(x3);
\foreach \x in {2,3}
{\draw[thick] (x0) to [bend left=30] (x\x);
\draw[thick] (x0) to [bend right=30] (x\x);}
\draw[thick] (x0) to [bend left=40] (x1);
\draw[thick] (x0) to [bend right=40] (x1);
\node[below] at (0,-10) {$M_3$};
\end{scope}

\begin{scope}[xshift=105cm]
\draw[fill=black] (0,0) {coordinate (x0)} circle (0.6);
\foreach \x in {1,2,3}
{\draw[fill=black] (\x*120-30:15) {coordinate (x\x)} circle (0.6);
\draw[thick] (x0)--(x\x);
\draw[thick] (x0) to [bend left=40] (x\x);
\draw[thick] (x0) to [bend right=40] (x\x);}
\node[below] at (0,-10) {$M_4$};
\end{scope}

\end{tikzpicture}    
\caption{Loopless multi-graphs $M_1$, $M_2$, $M_3$ and $M_4$.}\label{M1234}
\end{figure}

\begin{proof}
Suppose that $F$ contains a copy of $M_i$ for some $i=1,\ldots,7$ (see Figures \ref{M1234} and \ref{M567}). Notice that all vertices of $M_i$ have odd degrees. It follows that $F=M_i$ for some $i=1,\ldots,7$. Since $M_5,M_6,M_7$ has an edge-cut with two edges, by $F$ being 3-edge-connected, we further conclude that $F=M_1,M_2,M_3$ or $M_4$, and the assertion is true. So in the following we show that $F$ contains a copy of $M_i$ for some $i=1,\ldots,7$.

\begin{figure}[h]
\centering
\begin{tikzpicture}[scale=0.1]

\begin{scope}
\draw[fill=black] (0,0) {coordinate (x0)} circle (0.6);
\foreach \x in {1,2,3}
{\draw[fill=black] (\x*120-30:15) {coordinate (x\x)} circle (0.6);
\draw[thick] (x0)--(x\x);}
\draw[thick] (x0) to [bend left=40] (x1);
\draw[thick] (x0) to [bend right=40] (x1);
\draw[thick] (x2) to [bend left=20] (x3);
\draw[thick] (x2) to [bend right=20] (x3);
\node[below] at (0,-12) {$M_5$};
\end{scope}

\begin{scope}[xshift=35cm,yshift=3cm]
\foreach \x in {1,2,3,4}
{\draw[fill=black] (\x*90+45:15) {coordinate (x\x)} circle (0.6);}
\draw[thick] (x1)--(x4) (x2)--(x3);
\foreach \x/\y in {1/2,3/4}
{\draw[thick] (x\x) to [bend left=20] (x\y);
\draw[thick] (x\x) to [bend right=20] (x\y);}
\node[below] at (0,-15) {$M_6$};
\end{scope}

\begin{scope}[xshift=70cm,yshift=3cm]
\foreach \x in {1,2,3,4}
{\draw[fill=black] (\x*90+45:15) {coordinate (x\x)} circle (0.6);}
\foreach \x/\y in {1/2,3/4}
{\draw[thick] (x\x)--(x\y);
\draw[thick] (x\x) to [bend left=30] (x\y);
\draw[thick] (x\x) to [bend right=30] (x\y);}
\draw[thick] (x2) to [bend left=20] (x3);
\draw[thick] (x2) to [bend right=20] (x3);
\node[below] at (0,-15) {$M_7$};
\end{scope}

\end{tikzpicture}    
\caption{Loopless multi-graphs $M_5$, $M_6$ and $M_7$.}\label{M567}
\end{figure}
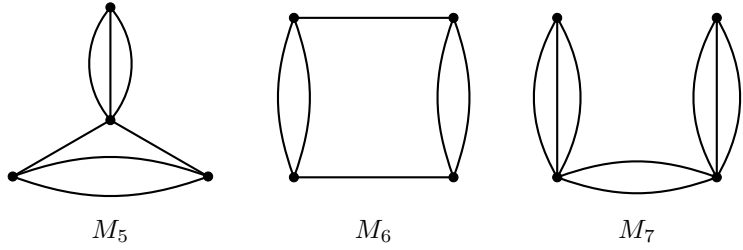

If $F$ has at least five vertices, then we choose a vertex $v\in V(F)$ such that $V(F)\backslash\{v\}$ has at least four vertices of odd degree. Let $e$ be an edge incident with $v$, and let $F'=F-e$. Since $F$ is 3-edge-connected, $F'$ is 2-edge-connected. But then there are more than two vertices of odd degree in $F'$,  contradicting our hypothesis. So we conclude that $F$ has exactly four vertices, and thus all vertices in $F$ are of odd degree.

Let $H$ be the underlying graph of $F$ (that is, be obtained from $F$ by merging all multi-edges). Clearly $H$ is connected, implying that $H=K_4, K_{1,1,2}, Z_1, C_4, K_{1,3}$ or $P_4$. For an edge $e\in E(H)$, we denote by $\mu(e)$ the multiple number of $e$ in $F$. For convenience, we let $v_1,v_2,v_3,v_4$ be the vertices of $F$ in non-increasing order of degree in $H$.

\underline{Case A. $H=K_4$.} In this case, $F$ contains a copy of $M_1$. 

\underline{Case B. $H=K_{1,1,2}$.} Since every vertex of $F$ has an odd degree, either $\mu(v_1v_3)\geq 2$ or $\mu(v_2v_3)\geq 2$. Assume without loss of generality that $\mu(v_1v_3)\geq 2$. We also have $\mu(v_1v_4)\geq 2$ or $\mu(v_2v_4)\geq 2$. If $\mu(v_1v_4)\geq 2$, then $F$ contains a copy of $M_2$; and if $\mu(v_2v_4)\geq 2$, then $F$ contains a copy of $M_6$. 

\underline{Case C. $H=Z_1$.} Since $v_1v_4$ is a cut-edge of $H$ and $F$ is 3-edge-connected, we have $\mu(v_1v_4)\geq 3$. If $\mu(v_2v_3)\geq 2$, then $F$ contains a copy of $M_5$. Now assume that $\mu(v_2v_3)=1$. It follows that $\mu(v_1v_2)\geq 2$ and $\mu(v_1v_3)\geq 2$, and $F$ contains a copy of $M_3$.

\underline{Case D. $H=C_4$.} Set $H=v_1v_2v_3v_4v_1$. If $\mu(v_1v_2)\geq 2$ and $\mu(v_3v_4)\geq 2$, then $F$ contains a copy of $M_6$. So we assume without loss of generality that $\mu(v_3v_4)=1$. It follows that $\mu(v_1v_4)\geq 2$ and $\mu(v_2v_3)\geq 2$, and again $F$ contains a copy of $M_6$.

\underline{Case E. $H=K_{1,3}$.} In this case every edge of $H$ is a cut-edge, and thus has multiple number at least 3. So $F$ contains a copy of $M_4$. 

\underline{Case F. $H=P_4$.} In this case every edge of $H$ is a cut-edge, and thus has multiple number at least 3. So $F$ contains a copy of $M_7$. 
\end{proof}

\subsection{The unfoldments of semi-loopless multi-graphs}

Let $F=F(e_0)$ be a multi-graph with a labeled edge $e_0$. We assume that $F$ has at least two edges, and all edges but possible $e_0$ are not loops. We call such a multi-graph $F$ with a labeled edge $e_0$ a \emph{semi-loopless} multi-graph. By a \emph{colored link} we mean a link in which every inner vertex is assigned red or blue, and the two end-vertices are assigned black.

For a semi-loopless graph $F=F(e_0)$, we define a colored link $L=L(x_0,y_0)$, also called an \emph{unfoldment} of $F$, as follows: Let $u_0,v_0$ be the two end-vertices of $e_0$ (possibly $e_0$ is a loop and $u_0=v_0$). First, for every edge $e\in E(F)\backslash\{e_0\}$ with end-vertices $u,v$, we assign two distinct vertices $u^e,v^e$ and a pure link $L^e$ connecting $u^e$ and $v^e$. Second, for every vertex $u\in V(F)$, we assign a clique $K^u$ with vertex set 
$$\{u^e: e\in E(F)\backslash\{e_0\}\mbox{ and }u\mbox{ is incident with }e\}\cup\{z: u=u_0\mbox{ and }z=x_0,\mbox{ or }u=v_0\mbox{ and }z=y_0\}.$$ 
Third, we assign the inner vertices of all pure links red and the end-vertices of all pure links blue (recall that the two vertices $x_0,y_0$ are black). Finally, let $L$ be the union of the above pure links and cliques $(\bigcup_{e\in E(F)\setminus\{e_0\}}L^e)\cup(\bigcup_{u\in V(F)}K^u)$, or obtained from this union by adding an edge $x_0y_0$. We remark that $F(e_0)$ has infinitely many of unfoldments. 

If $e_0$ is a loop, then both $x_0,y_0$ are contained in the clique $K^{u_0}$, clearly $x_0y_0\in E(L)$. 
If $e_0$ is not a loop, then $K^{u_0}$ and $K^{v_0}$ are two distinct cliques containing $x_0$, $y_0$, respectively. In this case $x_0y_0$ may be or not be an edge of $L$. 
We define the \emph{feasible paths} and \emph{associated vertices} of a colored link as that of a colored graph in Section 3.2.

\begin{lemma}\label{LeUnfoldFe0}
Let $F=F(e_0)$ be a semi-loopless multi-graph and let $L=L(x_0,y_0)$ be an unfoldment of $F$. Then\\
\indent (1) $L$ is claw-free;\\
\indent (2) all red vertices are of degree $2$;\\
\indent (3) every blue vertex has a unique red neighbor;\\
\indent (4) two associated blue vertices have no common blue neighbors or black neighbors;\\
\indent (5) if two associated blue vertices are adjacent, then they have a common red neighbor;\\
\indent (6) every black vertex has no red neighbors, and any two blue neighbors of a black vertex are adjacent.\\
Consequently, if a colored link $L$ satisfies (1)-(6), then it is an unfoldment of some semi-loopless multi-graph.
\end{lemma}

\begin{proof}
The neighbors of $x_0$ are those in $K^{u_0}$ or $y_0$, and the neighbors of $y_0$ are those in $K^{v_0}$ or $x_0$. This proves (6). The assertions (1)-(5) can be deduced similarly as in Lemma \ref{LeUnfold}.

Now we suppose that $L=L(x_0,y_0)$ is a colored link satisfying (1)-(6). Set $G'=L-\{x_0,y_0\}$. Then $G'$ satisfies (1)-(5). By Lemma \ref{LeUnfold}, there is a loopless multi-graph $F'$ such that 
$G'$ is an unfoldment of $F'$. In the following we will construct the desired 
semi-loopless multi-graph $F$
from $F'$ to complete the proof.

Let $V_b,V_r$ be the sets of blue vertices and red vertices of $L$, and we define $G'_b$ as the graph with the vertex set $V_b$ and the edge set $E(G_b')=\{xy\in E(G'): x,y 
\mbox{ are not associated}\}$.
As in the proof of Lemma \ref{LeUnfold}, we see that each component of $G'_b$ is a clique.

By (4) and (6), $x_0$ cannot have neighbors in distinct components of $G'_b$. We claim that if $x_0$ 
has a neighbor in a component of $G'_b$, then $x_0$ is adjacent to all vertices in this component.
 Suppose not, and let $H$ be a component of $G'_b$ such that $x_0$ adjacent to $x\in V(H)$ and non-adjacent to $y\in V(H)$. Let $\bar{x}$ be the unique red neighbor of $x$. Then $x_0\bar{x},y\bar{x}\notin E(G)$, and $\langle \{x,x_0,y,\bar{x}\}\rangle_L$ is an induced claw, contradicting (1).

Furthermore, if $x_0$ has some blue neighbors,
then there is a unique vertex $u_0\in V(F')$ such that $x_0$ is adjacent to 
those vertices $u_0^e$ for all edges $e\in E(F')$ incident with $u_0$.
In fact, those vertcies $u_0^e$ form a clqiue 
which is a component of $G'_b$.
If $x_0$ has no blue neighbors, then let $u_0$ be an extra vertex. 

Similarly, if $y_0$ has 
some blue neigbors, then there is a unique vertex $v_0\in V(F')$ such that $y_0$ is adjacent to those vertices $v_0^e$ for all edges $e\in E(F')$ incident with $v_0$. If $y_0$ has no blue neighbors, then let $v_0$ be an extra vertex.

Now let $F=F(e_0)$ be obtained from $F'$ by adding an labeled edge $e_0$ between $u_0,v_0$. Note that 
if exactly one of $\{u_0,y_0\}$ is not in $F'$
then $e_0$ is a pendant edge of $F$; if both $u_0$ and $v_0$
are not in $F'$ then $e_0$ is an isolated edge in $F$.
We can see that $L$ is an unfoldment of $F$.
\end{proof}

Note that $x_0$ may be not adjacent to $y_0$ in $L$. Let $L^{+}=L+x_0y_0$
if $x_0y_0\notin E(L)$, 
otherwise
$L^{+}=L$.

\begin{lemma}\label{LePropertyFe0L}
Let $F=F(e_0)$ be a semi-loopless multi-graph and let $L=L(x_0,y_0)$ be an unfoldment of $F$.
Then\\
\indent (1) $F$ is $2$-edge-connected if and only if $L^{+}$ is $2$-connected;\\
\indent (2) $F$ is $3$-edge-connected if and only if $L^{+}$ is $2$-connected and every $2$-cut $\{x,y\}$ of $L^{+}$ with both $x,y$ blue or black, is an associated pair;\\ 
\indent (3) $L$ has a Hamilton path from $x_0$ to $y_0$ if and only if $F$ has an Euler tour.
\end{lemma}

\begin{proof}
For every edge $e\in E(F)\backslash\{e_0\}$, let $P^e$ be the feasible path as in Lemma \ref{LePropertyFG}. We also set $P^{e_0}=x_0y_0$, and denote $u_0^{e_0}=x_0$, $v_0^{e_0}=y_0$. Now for every trail $T$ of $F$ from $u$ to $v$, we can get a path $P^T$ of $L^{+}$ from $u^{e_1}$ to $v^{e_k}$, where $e_1,e_k$ are the first and last edges of $T$, such that $P^T$ contains all paths $P^e$ for $e\in E(T)$. Moreover, if $T$ is a closed trail, then $L^{+}$ has a cycle $C^T$ containing all paths $P^e$ for $e\in E(T)$. Also notice that $\{x_0,y_0\}$ is not a 2-cut of $L^{+}$. Thus (1)(2) can be proved similarly as in Lemma \ref{LePropertyFG}. 

Now we prove (3). Suppose first that $F$ has an Euler tour $T$. Then $C^T$ is a Hamilton cycle of $L^{+}$ containing $x_0y_0$, and $P=C^T-x_0y_0$ is a Hamilton path of $L$ from $x_0$ to $y_0$.

Suppose now that $L$ has a Hamilton path $P$ from $x_0$ to $y_0$. Then $P$ contains all feasible paths of $L$, and we can assume that $P=x_0x_1P_1y_1x_2P_2y_2\ldots x_kP_ky_ky_0$, where $P_i$ is a feasible path connecting $x_i$ and $y_i$. By a similar analysis as in Lemma \ref{LePropertyFG}, we can prove that $F$ has a trail $T$ from $u_0$ to $v_0$ covering all edges but $e_0$, and $T+e_0$ is an Euler tour of $F$.
\end{proof}

\begin{lemma}\label{LeSubgraphFe0}
Let $F=F(e_0)$ be a semi-loopless multi-graph and let $L=L(x_0,y_0)$ be an unfoldment of $F$. If $F'$ is a proper subgraph of $F$ containing $e_0$, then $L$ has a proper induced subgraph $L'$ containing $x_0,y_0$, which is an unfoldment of $F'$.
\end{lemma}

\begin{proof}
The subgraph $L'$ is induced by $x_0,y_0$ and all the vertices of all the pure links $L^e$ for $e\in E(F')\backslash\{e_0\}$.
\end{proof}

\begin{lemma}\label{LeMinimalFe0}
Let $F=F(e_0)$ be a $3$-edge-connected semi-loopless multi-graph with at least two vertices of odd degree. If every $2$-edge-connected proper subgraph of $F$ containing $e_0$ has no vertices of odd degree, then $F$ is isomorphic to $N_1=N_1(e_0)$ or $N_2=N_2(e_0)$, as shown in Figure \ref{N12}.
\end{lemma}

\begin{figure}[h]
\centering
\begin{tikzpicture}[scale=0.1]

\begin{scope}
\draw[fill=black] (0,0) {coordinate (x1)} circle (0.6);
\draw[fill=black] (20,0) {coordinate (x2)} circle (0.6);
\draw[thick] (x1)--(x2);
\draw[thick] (x1) to [bend left=30] (x2);
\draw[thick] (x1) to [bend right=30] (x2);
\draw[thick] (25,0) circle (5);
\node[right] at (30,0) {$e_0$};
\node[below] at (15,-5) {$N_1(e_0)$};
\end{scope}

\begin{scope}[xshift=40cm]
\draw[fill=black] (0,0) {coordinate (x1)} circle (0.6);
\draw[fill=black] (20,0) {coordinate (x2)} circle (0.6);
\draw[thick] (x1)--(x2);
\draw[thick] (x1) to [bend left=30] (x2);
\draw[thick] (x1) to [bend right=30] (x2);
\node[above] at (10,3) {$e_0$};
\node[below] at (10,-5) {$N_2(e_0)$};
\end{scope}

\end{tikzpicture}    
\caption{Semi-loopless multi-graphs $N_1(e_0)$ and $N_2(e_0)$.}\label{N12}
\end{figure}
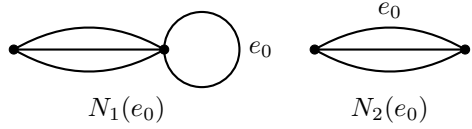

\begin{proof}
Suppose that $F$ contains a copy of $N_i$ for some $i=1,2$. Notice that both vertices of $N_i$ have odd degrees. It follows that $F=N_i$. So in the following we show that $F$ contains a copy of $N_i$ for some $i=1,2$.

If $F$ has at least three vertices, then we choose a vertex $v\in V(F)$ such that $V(F)\backslash\{v\}$ has at least two vertices of odd degree. Let $e$ be an edge incident to $v$ with $e\neq e_0$, and let $F'=F-e$. Since $F$ is 3-edge-connected, we have that $F'$ is 2-edge-connected. But then $F'$ contains some vertices of odd degree, 
contradicting our hypothesis. So we conclude that $F$ has exactly two vertices, and thus both vertices are of odd degree.

Let $V(F)=\{u_1,u_2\}$. Since $F$ is 3-edge-connected, there are at least three edges between $u_1,u_2$. If $e_0$ is a loop, say with the end-vertex $u_2$, then $F$ contains a copy of $N_1$. If $e_0$ is not a loop, then $u_1,u_2$ are the end-vertices of $e_0$, and $F$ contains a copy of $N_2$. 
\end{proof}

Let $\mathcal{L}_3(x,y)$ be the family of links from $x$ to $y$ that obtained from a link in $\mathcal{L}_2(x,y)$ by adding an edge $xy$, see Figure \ref{L3}. We remark that the unfoldments of $N_1(e_0)$ are the links in $\mathcal{L}_1(x_0,y_0)$, and the unfoldments of $N_2(e_0)$ are the links in $\mathcal{L}_2(x_0,y_0)\cup\mathcal{L}_3(x_0,y_0)$.

\begin{figure}[h]
\centering
\begin{tikzpicture}[scale=0.08]
\draw[fill=black] (-40,20){coordinate (a)} circle (0.8); \node[right] at (-41,16) {$a$};
\draw[fill=black] (-40,-20){coordinate (a')} circle (0.8); \node[right] at (-41,-16) {$a'$};
\draw[fill=black] (-20,10){coordinate (b)} circle (0.8); \node[right] at (-20,8) {$b$};
\draw[fill=black] (-20,-10){coordinate (b')} circle (0.8); \node[right] at (-20,-8) {$b'$};
\draw[fill=black] (-0,15){coordinate (x)} {node[right] {$x$}} circle (0.8);
\draw[fill=black] (-0,-15){coordinate (y)} {node[right] {$y$}} circle (0.8);
\draw[thick] (x)--(a)--(b)--(x); \draw[thick] (y)--(a')--(b')--(y);
\draw[thick] (x)--(y);
\draw[red,dotted,thick] (a)--(a') (b)--(b'); 
\end{tikzpicture}
\caption{Class of links $\mathcal{L}_3(x,y)$.}\label{L3}
\end{figure}
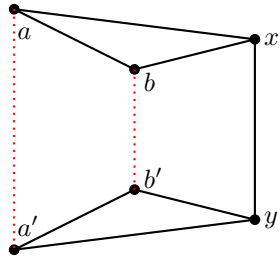

\section{Proof of Theorem \ref{Th3}}

Let $G$ be a minimal 2-connected non-cycle claw-free graph without spanning $\varTheta$-subgraphs. Obviously, $G$ is non-hamiltonian. We first show that $G$ has a 2-cut.

\begin{claim}\label{Cl2Connected}
$\kappa(G)=2$.
\end{claim}

\begin{proof}
Suppose that $G$ is 3-connected. Since $G$ is claw-free and non-hamiltonian, by Theorem~\ref{Th A}, there exists a vertex $x\in V(G)$ such that $\langle N(x)\rangle$ is disconnected. It follows that $\langle N(x)\rangle$ consists of two vertex-disjoint cliques. Let $A_1,A_2$ be the vertex set of the two cliques, and set $A=N(x)=A_1\cup A_2$. Let $H$ be a component of $G-x-A$. Since $G$ is 3-connected, $|N_A(H)|\geq 3$. For every vertex $z\in N_{A_i}(H)$, by $G$ being claw-free, we see that $N(z)\subseteq\{x\}\cup A_i\cup V(H)$ and $\langle N_H(z)\rangle$ is a clique. 

\underline{Case A. $H$ is non-separable.} First suppose that $N_A(H)\subseteq A_i$ for some $i=1,2$, say $A_1$. If $H$ is nontrivial, then there exist two vertices $z_1,z_2\in A_1$ such that $y_1z_1,y_2z_2\in E(G)$ for distinct $y_1,y_2\in V(H)$. If $H$ is trivial, then let $z_1,z_2$ be two arbitrary vertices in $N_{A_1}(H)$. Set $A'=N_{A_1}(H)\backslash\{z_1,z_2\}$. It follows that $G-A'$ is 2-connected and $\{z_1,z_2\}$ is a 2-cut of $G-A'$. Since $G$ is 3-connected, we see that $A'\neq\emptyset$. By the minimality of $G$, $G-A'$ contains a spanning $\varTheta$-subgraph $T$ (or $G-A'$ is a cycle, in which case we also denote $T=G-A'$). By Lemma \ref{LeCutxy}, $T$ contains an edge $z_1z'_1$ for some $z'_1\in\{x\}\cup A_1\backslash A'$. Notice that $\langle \{z_1,z'_1\}\cup A'\rangle$ is a clique. We can replace $z_1z'_1$ in $T$ with a path from $z_1$ to $z'_1$ passing through all vertices in $A'$ to obtain a spanning $\varTheta$-subgraph (or a Hamilton cycle) of $G$, a contradiction. 

Suppose now that $N_{A_1}(H)\neq\emptyset$ and $N_{A_2}(H)\neq\emptyset$. If $H$ is nontrivial, then there exist two vertices $z_1\in A_1$, $z_2\in A_2$ such that $y_1z_1,y_2z_2\in E(G)$ for distinct $y_1,y_2\in V(H)$. If $H$ is trivial, then let $z_1\in N_{A_1}(H)$, $z_2\in N_{A_2}(H)$ be arbitrary. Set $A'_i=N_{A_i}(H)\backslash\{z_i\}$, $i=1,2$, and $A'=A'_1\cup A'_2$. It follows that $G-A'$ is 2-connected and $\{z_1,z_2\}$ is a 2-cut of $G-A'$. By the minimality of $G$, $G-A'$ contains a spanning $\varTheta$-subgraph $T$ (or $G-A'=T$ is a cycle). By Lemma \ref{LeCutxy}, $T$ contains two edges $z_iz'_i$ for some $z'_i\in\{x\}\cup A_i\backslash A'_i$, $i=1,2$. Notice that $\langle \{z_i,z'_i\}\cup A'_i\rangle$ is a clique. We can replace $z_iz'_i$ in $T$ with a path from $z_i$ to $z'_i$ passing through all vertices in $A'_i$, $i=1,2$, to obtain a spanning $\varTheta$-subgraph (or a Hamilton cycle) of $G$, a contradiction. 

\underline{Case B. $H$ is separable.} First suppose that $H$ has an end-block $B_1$ such that there are no edges from the inner vertices of $B_1$ to $A_i$ for some $i=1,2$. Let $y_1$ be the cut-vertex of $H$ contained in $B_1$ and assume without loss of generality that $N_A(B_1-y_1)\subseteq A_1$. Let $z_1\in N_{A_1}(B_1-y_1)$ and $A'=N_{A_1}(B_1-y_1)\backslash\{z_1\}$. It follows that $G-A'$ is 2-connected and $\{y_1,z_1\}$ is a 2-cut of $G-A'$. By the minimality of $G$, $G-A'$ contains a spanning $\varTheta$-subgraph $T$ (or $G-A'=T$ is a cycle). By Lemma \ref{LeCutxy}, $T$ contains an edge $z_1z'_1$ for some $z'_1\in\{x\}\cup A_1\backslash A'$. By a similar analysis as above, we can find a spanning $\varTheta$-subgraph (or a Hamilton cycle) of $G$, a contradiction. 

Suppose now that for every end-block $B$ of $H$, there exist edges from some inner-vertices of $B$ to both $A_1$ and $A_2$. Let $B_i$, $i=1,2$, be two end-blocks of $H$, $y_i$ be the cut-vertex of $H$ contained in $B_i$, and let $z_i\in N_{A_i}(B_i-y_i)$. Note that every vertex in
 $A$ cannot have neighbors in both $B_1-y_1$ and $B_2-y_2$; for otherwise $G$ contains a claw. Set $A'_i=N_{A_i}(B_i-y_i)\backslash\{z_i\}$, $i=1,2$ and $A'=A'_1\cup A'_2$. It follows that $G-A'$ is 2-connected and $\{y_i,z_i\}$, $i=1,2$, are two 2-cuts of $G-A'$. By the minimality of $G$, $G-A'$ contains a spanning $\varTheta$-subgraph $T$ (or $G-A'=T$ is a cycle), and $T$ contains two edges $z_iz'_i$ for some $z'_i\in\{x\}\cup A_i\backslash A'_i$, $i=1,2$. By a similar analysis as above, we can find a spanning $\varTheta$-subgraph (or a Hamilton cycle) of $G$, a contradiction. 
\end{proof}

By Claim \ref{Cl2Connected}, $G$ has a 2-cut. For any 2-cut $\{x,y\} $ of $G$, we see that $G-\{x,y\}$ has exactly two components since $G$ is claw-free. Let $H_1,H_2$ be the two components of $G-\{x,y\}$. Note that each of $x,y$ has neighbors in both $H_1$ and $H_2$. The subgraph $L=L(x,y)$ induced by $\{x,y\}\cup V(H_1)$ is a link of $G$; and the subgraph induced by $\{x,y\}\cup V(H_2)$ is the co-link of $L$, denoted as $L^*=L^*(x,y)$. A link of $G$ from $x$ to $y$ is said to be \emph{simple} if it has a Hamilton path from $x$ to $y$; otherwise it is \emph{non-simple}. We note that at most one of the two links $L$ and $L^*$ is simple; otherwise $G$ will have a Hamilton cycle. We say that the 2-cut $\{x,y\}$ is a \emph{simple $2$-cut} if either $L$ or $L^*$ is simple; otherwise we say it a \emph{non-simple $2$-cut}. 

\begin{claim}\label{ClSimplePure}
Let $L=L(x,y)$ be a link of $G$. Then $L$ is a simple link of $G$ if and only if it is a pure link.
\end{claim}

\begin{proof}
If $L$ is a pure link, then clearly $L$ is simple. Now we assume that $L$ is a simple link of $G$, i.e., $L$ has a Hamilton path $P$ from $x$ to $y$. We will show that $L$ is either a triangle or a path from $x$ to $y$. Set $H_1=L-\{x,y\}$ and $H_2=G-L$.

First suppose that $xy\notin E(G)$. Let $Q$ be a shortest path of $L$ from $x$ to $y$. It follows that $Q$ is an induced path and has length at least 2. If $V(L)=V(Q)$, then $L$ is a path from $x$ to $y$ and we are done. So we conclude that $V(L)\backslash V(Q)\neq\emptyset$. Let $G'=G-V(L)\setminus V(Q)$. Then $G'$ is 2-connected. It follows that $G'$ contains a spanning $\varTheta$-subgraph $T$ (or $G'=T$ is a cycle). Since all internal vertices of $Q$ have degree 2 in $G'$, we see that $T$ contains $Q$. Now we can replace $Q$ with the path $P$ in $T$ to obtain a spanning $\varTheta$-subgraph (or a Hamilton cycle) of $G$, a contradiction. 

Suppose now that $xy\in E(G)$. We claim that $x,y$ have a common neighbor in $H_1$. Suppose not the case. Let $\bar{x}$ be a neighbor of $x$ in $H_1$, and $\bar{y}$ be a neighbor of $y$ in $H_1$ such that $x\bar{y},y\bar{x}\notin E(G)$. For any vertex $v\in V(H_2)$, we see that $vx\in E(G)$ if and only if $vy\in E(G)$; for otherwise $\langle \{x,\bar{x},y,v\}\rangle$ or $\langle \{y,\bar{y},x,v\}\rangle$ is a claw. Thus we have that $N_{H_2}(x)=N_{H_2}(y)$. Note that $\langle N_{H_2}(x)\rangle$ is a clique. 
Let $G'=G-(V(H_1)\cup\{y\})$. Clearly $G'$ is 2-connected. By the minimality of $G$, $G'$ contains a spanning $\varTheta$-subgraph $T$ (or $G'=T$ is a cycle). So $T$ contains an edge $xv$ with $v\in N_{H_2}(x)$. Since $yv\in E(G)$, we can replace $xv$ with the path $xPyv$ in $T$ to obtain a spanning $\varTheta$-subgraph (or a Hamilton cycle) of $G$, a contradiction. So we conclude that $x,y$ have a common neighbor in $H_1$, say $z$.

If $H_1$ has only one vertex $z$, then $L$ is a triangle $xyzx$, and we are done. So we assume that $V(H_1)\backslash\{z\}\neq\emptyset$. Let $G'=G-V(H_1)\backslash\{z\}$. Then $G'$ is 2-connected. It follows that $G'$ contains a spanning $\varTheta$-subgraph $T$. Since $z$ has degree 2 in $G'$, we see that $xz,zy\in E(T)$. Now we can replace $xzy$ with the path $P$ in $T$ to obtain a spanning $\varTheta$-subgraph of $G$, a contradiction.
\end{proof}

A simple link $L$ of $G$ is called a \emph{maximal} simple link if there are no other simple links containing $L$; and a non-simple link $L$ of $G$ is called a \emph{minimal} non-simple link if all other links contained in $L$ are simple. Let $V_2$ be the set of vertices of $G$ of degree 2, and let $V_{\geq 3}$ be the set of vertices of $G$ of degree at least 3.

\begin{claim}\label{ClDegree23}
Every vertex in $V_2$ is an inner vertex of a simple link of $G$; every vertex in $V_{\geq 3}$ is contained in a 2-cut of $G$.
\end{claim}

\begin{proof}
If $z\in V_2$, say $N(z)=\{x,y\}$, then clearly $\{x,y\}$ is a 2-cut of $G$, $\langle \{x,y,z\}\rangle:=L(x,y)$ is a simple link of $G$ and $z$ is an inner vertex of the link. This proves the first assertion of the claim.

Now we shall prove that every vertex in $V_{\geq 3}$ is contained in a 2-cut of $G$. For this we let $A_1$ be the set of vertices in $V_{\geq 3}$ that are contained in some 2-cuts; and $A_2$ be the set of vertices in $V_{\geq 3}$ that are not contained in any 2-cut. It suffices  to show that $A_2=\emptyset$. Suppose otherwise that $A_2\neq\emptyset$.

If $z\in A_2$ has a neighbor $\bar{z}\in V_2$, then $N(\bar{z})$ is a 2-cut of $G$ containing $z$, contradicting  $z\in A_2$. So we conclude that there are no edges between $A_2$ and $V_2$. 
Since $G$ is 2-connected, there are some edges between $A_1$ and $A_2$. Let $xz\in E(G)$ with $x\in A_1$ and $z\in A_2$. 

Let $\{x,y\}$ be a 2-cut of $G$ (since $x\in A_1$) and let $H$ be the component of $G-\{x,y\}$ containing $z$. Since $G$ is claw-free, $\langle N_{H}(x)\rangle$ is a clique containing $z$. By the assumption that $z\in A_2$, we have that $G-z$ is 2-connected. It follows that $G-z$ contains a spanning $\varTheta$-subgraph $T$ (or $G-z=T$ is a cycle). 
By Lemma \ref{LeCutxy}, $T$ contains an edge $xz_1$ with $z_1\in N_{H}(x)\backslash\{z\}$. Notice that $zz_1\in E(G)$. Now we can replace $xz_1$ with $xzz_1$ in $T$ to obtain a spanning $\varTheta$-subgraph (or a Hamilton cycle) of $G$, a contradiction.
\end{proof}

It follows from Claims \ref{ClSimplePure} and \ref{ClDegree23}
that 
 a vertex is an inner vertex of a simple link of $G$ if and only if it is in $V_2$. 

\begin{claim}\label{ClDisjointLink}
Any two maximal simple links of $G$ are vertex-disjoint.
\end{claim}

\begin{proof}
Suppose not.
Let $L_1=L_1(x_1,y_1)$ and $L_2=L_2(x_2,y_2)$ be two maximal simple links of $G$ such that $V(L_1)\cap V(L_2)\neq\emptyset$. We can see that $x_1,x_2,y_1,y_2\in V_{\geq 3}$. It follows that the only possible common vertex of $L_1,L_2$ are their end vertices. We may assume without loss of generality that $y_1=x_2$. Since $G$ is claw-free, $N(y_1)=N_{L_1}(y_1)\cup N_{L_2}(x_2)$. Now $\{x_1,y_2\}$ is a 2-cut of $G$, and $L=L_1\cup L_2$ is a link of $G$ from $x_1$ to $y_2$. 

Notice that a Hamilton path of $L_1$ from $x_1$ to $y_1$ and a Hamilton path of $L_2$ from $x_2$ to $y_2$ form a Hamilton path of $L$ from $x_1$ to $y_2$. This implies that $L$ is a simple link containing both $L_1$ and $L_2$, contradicting that $L_1$ and $L_2$ are maximal.
\end{proof}

Now we distinguish the following two cases.

\begin{case}
Every 2-cut of $G$ is simple.
\end{case}

\begin{claim}\label{ClCommonNeighborGL}
Let $L=L(x,y)$ be a maximal simple link of $G$. Then $x,y$ cannot have  common neighbors in $G-V(L)$.
\end{claim}

\begin{proof}
Since $L$ is a maximal simple link of $G$, both $x$ and $y\in V_{\geq 3}$. Suppose that $x,y$ have a common neighbor $z$ in $G-V(L)$. We have $z\in V_{\geq 3}$; otherwise
$z\in V_2$ and 
$x$ is contained in two maximal simple links.

By Claim \ref{ClDegree23}, $z$ is contained in a 2-cut of $G$, and 
then is an end-vertex of a unique maximal simple link of $G$ by Claim 4. Let $\bar{z}$ be a neighbor of $z$ in $V_2$. It follows that $\bar{z}\notin V(L)$ and thus $x\bar{z},y\bar{z}\notin E(G)$. Since $G$ is claw-free, $xy\in E(G)$, implying that $L$ is a triangle, say $L=xywx$. 

We claim that either $N_{G-L}(x)\subseteq N_{G-L}(y)$ or $N_{G-L}(y)\subseteq N_{G-L}(x)$. Suppose otherwise that there are vertices $x_1,y_1$ outside $L$ such that $xx_1,yy_1\in E(G)$ and $xy_1,yx_1\notin E(G)$. 
Since each of $\{x_1,y_1\}$ is  not contained in two maximal simple links of $G$, it follows that $x_1,y_1\in V_{\geq 3}$, and $x_1w,zw\notin E(G)$, implying that $x_1z\in E(G)$, and similarly, $y_1z\in E(G)$. 
Recall that 
$z$ is an end-vertex of a unique maximal simple link of $G$.
It follows that $x_1\bar{z}\notin E(G)$ or $y_1\bar{z}\notin E(G)$. Without loss of generality we assume that $x_1\bar{z}\notin E(G)$, which implies that $yx_1\in E(G)$ by $G$ being claw-free, a contradiction. So we can assume without loss of generality that $N_{G-L}(x)\subseteq N_{G-L}(y)$.

Let $G'=G-\{w,y\}$. Recall that $\langle N_{G-L}(y)\rangle$ is a clique of $G$, we see that $G'$ is 2-connected. It follows that $G'$ has a spanning $\varTheta$-subgraph $T$ (or $G'=T$ is a cycle). Let $xx_1$ be an edge of $T$. Then $yx_1\in E(G)$ and we can replace $xx_1$ with $xwyx_1$ in $T$ to obtain a spanning $\varTheta$-subgraph (or a Hamilton cycle) of $G$, a contradiction. 
\end{proof}

We color all vertices in $V_2$ red and all vertices in $V_{\geq 3}$ blue. We now show that $G$ satisfies the conditions (1)-(5) of Lemma \ref{LeUnfold}. The conditions (1)(2) are trivial. If there exists a vertex $x\in V_{\geq 3}$ such that $x$ has  two neighbors in $V_2$, then $x$ will be contained in two maximal simple links of $G$, contradicting Claim \ref{ClDisjointLink}. This proves (3). Let $x,y$ be two associated vertices of $G$. Then $x,y$ are the end-vertices of a maximal simple link $L=L(x,y)$ of $G$, and (4) was proved in Claim \ref{ClCommonNeighborGL}. If $xy\in E(G)$, then $L$ is a triangle by Claim \ref{ClSimplePure}, implying that $x,y$ have a common neighbor in $V_2$. This proves (5). By Lemma \ref{LeUnfold}, $G$ is an unfoldment of a multi-graph $F$.

Recall that $G$ is 2-connected. If $G$ has a 2-cut $\{x,y\}$ such that $x,y\in V_{\geq 3}$ are not associated, then the two links $L=L(x,y)$ and $L^*=L^*(x,y)$ are not pure links, and then not simple links by Claim \ref{ClSimplePure}. It follows that $\{x,y\}$ is a non-simple 2-cut of $G$, a contradiction. So we conclude that every 2-cut of $G$ with $x,y\in V_{\geq 3}$ is an associated pair. By Lemma \ref{LePropertyFG}, $F$ is 3-edge-connected. 

If $F$ has at most two vertices of odd degree, then $F$ has an Euler trail, and $G$ has a spanning $\varTheta$-subgraph by Lemma \ref{LePropertyFG}, a contradiction. So we conclude that $F$ has at least four vertcies of odd degree. If $F$ has a 2-edge-connected proper subgraph $F'$ with at least four vertcies of odd degree, then $G$ has a proper induced subgraph $G'$ which is an unfoldment of $F'$ by Lemma \ref{LeSubgraphF}. Clearly $G'$ is not a cycle, and by Lemma \ref{LePropertyFG}, $G'$ has no spanning $\varTheta$-subgraph, a contradiction. So we conclude that every 2-edge-connected proper subgraph of $F$ has at most two vertices of odd degree. 

By Lemma \ref{LeMinimalF}, $F$ is isomorphic to one of $M_1,M_2,M_3,M_4$. It follows that $G\in\mathcal{H}_1\cup\mathcal{H}_2\cup\mathcal{H}_3\cup\mathcal{H}_4$.

\begin{case}
 $G$ has a 
non-simple 2-cut $\{x_0,y_0\}$.
\end{case}

Then the two links $L_0=L_0(x_0,y_0)$ and $L_0^*=L_0^*(x_0,y_0)$ are non-simple. Let $L_1=L_1(x_1,y_1)$ be a minimal non-simple link contained in $L_0$, and $L_2=L_2(x_2,y_2)$ be a minimal non-simple link contained in $L_0^*$ (possibly $\{x_0,y_0\}=\{x_1,y_1\}$ or $\{x_2,y_2\}$). So $L_1$ and $L_2$ are internally-disjoint.

\begin{claim}\label{ClNotCutxy}
Let $x\in V(L_1)\backslash\{x_1,y_1\}$ and $y\in V(L_1^*)$. Then $\{x,y\}$ is not a 2-cut of $G$.
\end{claim}

\begin{proof}
By contradiction, suppose that $\{x,y\}$ is a
2-cut of $G$.
If $y\in\{x_1,y_1\}$, then 
we assume that $\{x,y_1\}$ is a 2-cut of $G$ without loss of generality.
Let $L'_1=L'_1(x,y_1)$ be a link contained in $L_1$. Since $L_1$ is a minimal non-simple link, we have that $L'_1$ is simple. Let $P'_1$ be the Hamilton path of $L'_1$ from $x$ to $y_1$. 
If $V(L_1)\backslash V(L'_1)=\{x_1\}$, then $xx_1\in E(G)$, and $x_1xP'_1y_1$ is a Hamilton path of $L_1$ from $x_1$ to $y_1$, contradicting that $L_1$ is non-simple. So we conclude that $V(L_1)\backslash(\{x_1\}\cup V(L'_1))\neq\emptyset$. 

By $G$ being claw-free, $N(y_1)\subseteq V(L'_1)\cup V(L_1^*)$. It follows that every path from $L'_1-\{x,y_1\}$ to $L_1-V(L'_1)\cup \{x_1\}$ will pass through either $x$ or $x_1$. We have that $\{x,x_1\}$ is also a 2-cut of $G$. Now $L''_1=L''_1(x,x_1)$ is a link contained in $L_1$, which is simple. However, $L'_1$ and $L''_1$ intersect at $x$, contradicting Claim \ref{ClDisjointLink}. So we conclude that $\{x,y_1\}$ is not a 2-cut of $G$, and similarly, $\{x,x_1\}$ is not a 2-cut of $G$.

Hence $y\in V(L_1^*)\backslash\{x_1,y_1\}$. Then $\{x_1,y_1\}$ separates $x,y$.
Since $\{x,y\}$ is a 2-cut of $G$,
it follows that 
 $x_1,y_1$ are in distinct components of $G-\{x,y\}$. If $L_1$ has only one inner-vertex $x$, then $L_1$ is simple, a contradiction. So we let $z$ be an inner-vertex of $L_1$ other than $x$. Assume without loss of generality that $x_1$ and $z$ are in distinct components of $G-\{x,y\}$. Then every path from $x_1$ to $z$ will pass through either $x$ or $y_1$. This implies that $\{x,y_1\}$ is a 2-cut of $G$, by the previous case, we again get a  contradiction.
\end{proof}

Let $V'_2$ be the set of the inner vertices of $L_1$ of degree 2, and let $V'_{\geq 3}$ be the set of inner vertices of $L_1$ of degree at least 3. By Claim \ref{ClDegree23}, every vertex in $V'_{\geq 3}$ is contained in a 2-cut. By Claim \ref{ClNotCutxy}, such a 2-cut is contained in $V(L_1)\backslash\{x_1,y_1\}$. It follows that every vertex in $V'_{\geq 3}$ is an end-vertices of a maximal simple link contained in $L_1-\{x_1,y_1\}$, which is simple by the choice of $L_1$. By Claims \ref{ClDegree23} and \ref{ClDisjointLink}, $V'_{\geq 3}$ can be partitioned into pairs such that each pair of vertices are connected by a maximal simple link contained in $L_1$.

If $x_1$ has a neighbor in
$V'_2$, then $x_1$ is an end-vertex of a simple link contained in $L_1$,
thus $x_1$ and some vertex of this simple link form a 2-cut of $G$,
contradicting Claim \ref{ClNotCutxy}. So we conclude that $x_1$, and similarly, $y_1$ has no neighbors in $V'_2$. 

By coloring the vertices in $V'_2$ red, the vertices in $V'_{\geq 3}$ blue, and the vertices $x_1,y_1$ black, we get a colored link $L_1$. Notice that the neighbors of $x_1$ (or $y_1$) in $L_1-\{x_1,y_1\}$ form a clique. We see that the conditions (1)-(6) of Lemma \ref{LeUnfoldFe0} hold. It follows that $L_1$ is an unfoldment of a semi-loopless multi-graph $F=F(e_1)$. 

\begin{claim}
$F$ is 3-edge-connected; $F$ has at least two vertices of odd degree; and every 2-edge-connected proper subgraph $F'$ of $F$ containing $e_1$ has no vertices of odd degree.
\end{claim}

\begin{proof}
Clearly $L_1+x_1y_1$ is 2-connected and $\{x_1,y_1\}$ is not a 2-cut of $L_1$. If $L_1+x_1y_1$ has a 2-cut $\{x,y\}$ with $x,y\in\{x_1,y_1\}\cup V_{\geq 3}$, then $\{x,y\}$ is also a 2-cut of $G$. It follows from Claim~6 that 
$\{x,y\}\cap\{x_1,y_1\}=\emptyset$, and then $x,y$ are connected by a link of $G$ contained in $L_1$, which is a simple link. That is, $x,y$ are associated. By Lemma \ref{LePropertyFe0L}, $F$ is 3-edge-connected. 

Recall that $L_1$ is non-simple, i.e., it has no Hamilton paths between $x_1$ and $y_1$. By Lemma \ref{LePropertyFe0L}, $F$ has no Euler tours, and then has at least two odd degree vertices. Suppose now that $F$ has a 2-edge-connected proper subgraph $F'$ containing $e_1$ with at least two odd degree vertices. By Lemma \ref{LeSubgraphFe0}, $L_1$ has a proper induced subgraph $L'_1=L'_1(x_1,y_1)$, which is an unfoldedment of $F'$. 
Since $F'$
has no Euler tours,
and by Lemma \ref{LePropertyFe0L}, $L'_1$ has no Hamilton paths between $x_1$ and $y_1$. Since $F'$ is 2-edge-connected, $L'_1+x_1y_1$ is 2-connected.

Let $G'=G-V(L_1)\setminus V(L'_1)$. Then $G'$ is 2-connected, and thus $G'$ contains a spanning $\varTheta$-subgraph $T$. Notice that $L'_1$ and $L_2$ are two internally-disjoint links of $G'$. It follows that $T$ contains either a Hamilton path of $L'_1$ from $x_1$ to $y_1$, or a Hamilton path of $L_2$ from $x_2$ to $y_2$, a contradiction. 
\end{proof}

By Lemma \ref{LePropertyFe0L}, $F=N_1(e_1)$ or $N_2(e_1)$. It follows that $L_1\in\mathcal{L}_1\cup\mathcal{L}_2\cup\mathcal{L}_3$. Similarly, we can prove that $L_2\in\mathcal{L}_1\cup\mathcal{L}_2\cup\mathcal{L}_3$.

Set $H=L_1^*\cap L_2^*$. We will show that $H$ is a chain from $\{x_1,y_1\}$ to $\{x_2,y_2\}$ and together with the possible edges $x_1y_1,x_2y_2$. Since $G$ is 2-connected, $G$ has two vertex-disjoint paths from $\{x_1,y_1\}$ to $\{x_2,y_2\}$. We choose such two paths $P_1,P_2$ with $|P_1|+|P_2|$ as small as possible. We suppose without loss of generality that $P_1$ is from $x_1$ to $y_1$ and $P_2$ is from $x_2$ to $y_2$. Notice that both $P_1,P_2$ are induced paths contained in $H$.

\begin{claim}\label{ClVHP1P2}
$V(H)=V(P_1)\cup V(P_2)$.
\end{claim}

\begin{proof}
Suppose, otherwise, that $V(H)\backslash(V(P_1)\cup V(P_2))\neq\emptyset$. Let $G'=G-V(H)\setminus (V(P_1)\cup V(P_2))$. Then $G'$ is a 2-connected proper induced subgraph of $G$. It follows that $G'$ has a spanning $\varTheta$-subgraph $T$. Notice that $T$ contains either a Hamilton path of $L_1$ from $x_1$ to $y_1$, or a Hamilton path of $L_2$ from $x_2$ to $y_2$. Thus either $L_1$ or $L_2$ is a simple link of $G$, a contradiction.  
\end{proof}

Two edges $u_1v_1$ and $u_2v_2$ from $P_1$ to $P_2$ are \emph{crossed} if $u_1$ appears before $u_2$ in $P_1$ and $v_2$ appears before $v_1$ in $P_2$ (or $u_1$ appears after $u_2$ in $P_1$ and $v_2$ appears after $v_1$ in $P_2$). For every vertex $v$ in $P_1$ or $P_2$, we denote $v^-$ and $v^+$ the predecessor and successor of $v$ along $P_1$ or $P_2$. For the case $v\in\{x_1,y_1,x_2,y_2\}$, we let $x_1^-,y_1^-$ be two neighbors of $x_1,y_1$ in $L_1-\{x_1,y_1\}$, respectively; and $x_2^+,y_2^+$ be two neighbors of $x_2,y_2$ in $L_2-\{x_2,y_2\}$, respectively.

\begin{claim}\label{ClCrossed}
If $u_1v_1$ and $u_2v_2$ are two crossed edges from $P_1$ to $P_2$, then $u_2=u_1^+$ and $v_2=v_1^-$ (or $u_2=u_1^-$ and $v_2=v_1^+$), and $u_1v_2,u_2v_1\in E(G)$. Consequently,
$\langle\{u_1,u_2,v_1,v_2\}\rangle_G$
is a 4-clique.
\end{claim}

\begin{proof}
Suppose that either $|P_1[u_1,u_2]|\geq 2$ or $|P_2[v_2,v_1]|\geq 2$. Let $P'_1=P_1[x_1,u_1]u_1v_1P_2[v_1,y_2]$ and $P'_2=P_2[y_1,v_2]v_2u_2P_1[u_2,x_2]$. It follows that $P'_1,P'_2$ are two vertex-disjoint paths from $\{x_1,y_1\}$ to $\{x_2,y_2\}$ with $|P'_1|+|P'_2|<|P_1|+|P_2|$, contradicting the choices of $P_1$ and $P_2$. Thus we conclude that $u_2=u_1^+$ and $v_2=v_1^-$.

Recall that $P_1,P_2$ are induced paths. We have $u_1u_2^+\notin E(G)$. If $u_2^+v_2\in E(G)$, then $u_1v_1$ and $u_2^+v_2$ are crossed and $|P_1[u_1,u_2^+]|\geq 2$, a contradiction. Thus we have $u_2^+v_2\notin E(G)$, and then $u_1v_2\in E(G)$, since otherwise 
$\langle \{u_1,u_2,u_2^+,v_2\}\rangle$
is a claw. Similarly we can prove that $u_2v_1\in E(G)$.
\end{proof}

By Claim \ref{ClCrossed}, the only possible edge crossed with $u_1v_1$ is either $u_1^-v_1^+$ or $u_1^+v_1^-$, but not both. It follows that every edge from $P_1$ to $P_2$ is crossed with at most one other edges. We label the vertices $x_1=a_1,c_1,a_2,c_2,\ldots,a_k,c_k=x_2\in V(P_1)$, $y_1=b_1,d_1,b_2,d_2,\ldots,b_k,d_k=y_2\in V(P_2)$, such that $c_ib_{i+1},d_ia_{i+1}$, $i=1,\ldots,k-1$, are all the crossed pairs from $P_1$ to $P_2$. We assume that $a_1,c_1,a_2,c_2,\ldots,a_k,c_k$ appear in this order along $P_1$ and $b_1,d_1,b_2,d_2,\ldots,b_k,d_k$ appear in this order along $P_2$. 

By Claim \ref{ClCrossed}, $a_{i+1}=c_i^+$, $b_{i+1}=d_i^+$, and $\langle \{c_i,d_i,a_{i+1},b_{i+1}\}\rangle$ is a 4-clique for $i=1,\ldots,k-1$. We remark that possibly $a_i=c_i$ or $b_i=d_i$. Set $H_i=\langle V(P_1[a_i,c_i])\cup V(P_2[b_i,d_i])\rangle$. Now $H_i$ contains no edges from $P_1$ to $P_2$ that are crossed with some edges. This implies that 
$H_i$ cannot contain a $4$-clique.

\begin{claim}\label{ClPureChainHi}
$H_i$ is a pure chain from $\{a_i,b_i\}$ to $\{c_i,d_i\}$ which possibly contains the edges
$a_ib_i,c_id_i$.      
\end{claim}

\begin{proof}
If there are no edges from $P_1[a_i,c_i]$ to $P_2[b_i,d_i]$ other than the possible edges $a_ib_i,c_id_i$, then $H_i-\{a_ib_i,c_id_i\}$ is a bipath chain, and we are done. Now suppose there is an edge $uv$ from $P_1[a_i,c_i]$ to $P_2[b_i,d_i]$ other than the possible edges $a_ib_i,c_id_i$. 
 We can choose $uv$ such that there are no edges from $P_1[a_i,u]$ to $P_2[b_i,v]$ other than $a_ib_i$ and $uv$. That is, $uv$ is the first edge 
 from $P_1[a_i,c_i]$ to $P_2[b_i,d_i]$
except for the possible edge $a_ib_i$.

We claim that either $u^-v\in E(G)$ or $uv^-\in E(G)$. Suppose that $u^-v,uv^-\notin E(G)$. Notice that $u^-v^+\notin E(G)$ and $u^+v^-\notin E(G)$; for otherwise they will be crossed with $uv$. Thus $uv^+\in E(G)$ and $u^+v\in E(G)$ for otherwise 
$\langle \{v^-,v,v^+,u\} \rangle $ or
$\langle \{u^-,u,u^+,v\} \rangle $
is a claw. So  $uv^+$ crosses $u^+v$.
Since $H_i$ contains no crossing pairs of edges 
from $P_1$ to $P_2$, it follows that 
$uv=c_id_i$, contradicting
the choice of $uv$.

Without loss of generality we assume that $u^-v\in E(G)$. We here show that $u\neq a_i$. Suppose otherwise that $u=a_i$, and thus $v\neq b_i$. If $i=1$, then $u=x_1$ and $u^-\in V(L_1)\backslash\{x_1,y_1\}$, which is nonadjacent to any vertex in $V(P_2)\backslash\{y_1\}$, a contradiction. If $i\geq 2$, then $u^-=c_{i-1}$, and thus $u^-v$ and $a_ib_i$ are crossed, a contradiction. Now we conclude that $u\neq a_i$, and $u^-\in V(H_i)$. By the choice of $uv$, we have that $u^-v=a_ib_i$, i.e., $u^{-}=a_i$, $v=b_i$. 

We have that that $uv^-\notin E(G)$; for otherwise, it is crossed with $a_ib_i$. By the claw-freeness of $G$, we have $uv^+\in E(G)$. Now we show that $v^+\in V(H_i)$. If $v^+\notin V(H_i)$, then $v=d_i$, implying that $u\neq c_i$. It follows that $uv^+$ is crossed with $c_id_i$, a contradiction.

Set $P_1[a_i,c_i]=u_1u_2\ldots u_s$ and $P_2[b_i,d_i]=v_1v_2\ldots v_t$,
where $u_1=u^-=a_i,v_1=v=b_i,u_2=u,v_2=v^+$.
We already know that $u_1v_1,u_2v_1,u_2v_2\in E(G)$. By a similar progress, we can prove that: (1) if $u_{j}v_{j-1}\neq c_id_i$, then $u_jv_j\in E(G)$ and $v_j\in V(H_i)$, and (2) if $u_jv_j\neq c_id_i$, then $u_{j+1}v_{j}\in E(G)$ and $u_{j+1}\in V(H_i)$. It follows that $H_i$ is a triangle chain from $\{a_i,b_i\}$ to $\{c_i,d_i\}$.
\end{proof}

By Claim \ref{ClPureChainHi}, we see that $H$ is a chain from $\{x_1,y_1\}$ to $\{x_2,y_2\}$ which possibly contains the edges $x_1y_1,x_2y_2$. 
Assume first that $x_1y_1\notin E(G)$. Then $L_1\in\mathcal{L}_2$, and $H_1$ is a bipath chain. Moreover, if $H_1$ is a trivial chain, then $k=1$, and both $L_1$ and $L_2$ are in $\mathcal{L}_2$. For the case $x_1y_1\in E(G)$, we prove the following.

\begin{claim}\label{Clx1y1}
If $x_1y_1\in E(G)$, then either (1) $L_1\in\mathcal{L}_1$; or (2) $H_1-x_1y_1$ is a trivial chain and $k\geq 2$; or (3) $H_1-x_1y_1$ is a trivial chain, $k=1$, and $L_2\in\mathcal{L}_1$.
\end{claim}

\begin{proof}
Since $x_1y_1\in E(G)$, we have $L_1\in\mathcal{L}_1\cup\mathcal{L}_3$. If $L_1\in\mathcal{L}_1$, then (1) holds. So we assume that $L_1\in\mathcal{L}_3$. 

By Claim~10, $H_1$ is a pure chain.
Furthermore, if $H_1$ is a triangle chain, then either $x_1y_1^+\in E(G)$ or $x_1^+y_1\in E(G)$ but not both. Without loss of generality we assume that $x_1y_1^+\in E(G)$ and $x_1^+y_1\notin E(G)$. It follows that $\langle \{x_1,x_1^-,y_1,x_1^+\}\rangle$ is a claw, a contradiction. 

Suppose that $H_1$ is a nontrivial bipath chain which contains the edge $x_1y_1$ and the possible edge $c_1d_1$. Without loss of generality we assume that $P_2[y_1,d_1]$ is nontrivial, then $x_1^+y_1\notin E(G)$
since $H_1$ is a nontrivial bipath chain.
Now $\langle \{x_1,x_1^-,y_1,x_1^+\}\rangle$ is a claw, a contradiction. 

So we conclude that $H_1$ is a trivial chain which contains the edge $x_1y_1$. If $k\geq 2$, then (2) holds. So we assume that $k=1$, i.e., $x_1=x_2$, $y_1=y_2$. If $L_2\in\mathcal{L}_3$ as well, then $\langle \{x_1,x_1^-,y_1,x_1^+\}\rangle$ is a claw, a contradiction. So we conclude that $L_2\in\mathcal{L}_1$ and (3) holds.    
\end{proof}

We note that the  edge $x_1y_1$ in Claim \ref{Clx1y1} can be considered as an edge in $L_1$ for (1), an edge in 4-clique $\langle \{c_1,d_1,a_2,b_2\}\rangle$ for (2), and an edge in $L_2$ for (3). Thus we always regard $L_1$ a link in $\mathcal{L}_1\cup\mathcal{L}_2$. It follows that if $L_1\in\mathcal{L}_2$, then $H_1$ is a bipath chain. By symmetry, we can prove that if $L_2\in\mathcal{L}_2$, then $H_k$ is a bipath chain. Thus $G\in\mathcal{H}_5\cup\mathcal{H}_6\cup\mathcal{H}_7$.

The proof is complete.

\section{Proofs of Observation~\ref{Th1} and Theorem~\ref{Th2}}

We construct nine 2-connected non-cycle graphs $G_1,\cdots,G_9$ without spanning $\varTheta$-subgraphs, as shown in Figure \ref{G12345678}.

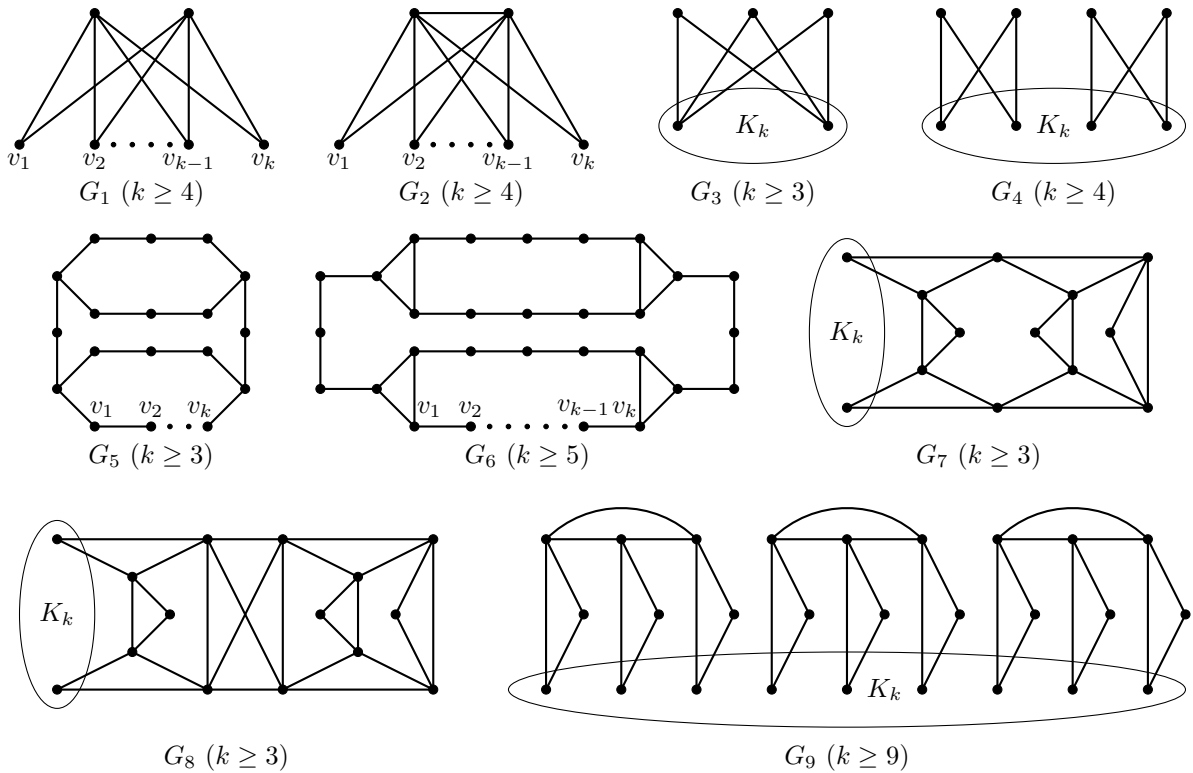
\begin{figure}[h]
\centering
\begin{tikzpicture}[scale=0.05]

\begin{scope}[xshift=-122.5cm,yshift=150cm]
\draw[fill=black] (-32.5,-5) {coordinate (v1)} {node[below] {$v_1$}} circle (1.2);
\draw[fill=black] (-12.5,-5) {coordinate (v2)} {node[below] {$v_2$}} circle (1.2);
\draw[fill=black] (12.5,-5) {coordinate (v3)} {node[below] {$v_{k-1}$}} circle (1.2);
\draw[fill=black] (32.5,-5) {coordinate (v4)} {node[below] {$v_k$}} circle (1.2);
\draw[fill=black] (-12.5,30) {coordinate (u1)} circle (1.2);
\draw[fill=black] (12.5,30) {coordinate (u2)} circle (1.2);
\foreach \x in {0,1,...,5} \draw[fill=black] (\x*5-12.5,-5) circle (0.6);
\foreach \x in {1,2} \foreach \y in {1,2,3,4} \draw[thick] (u\x)--(v\y);
\node[below] at (0,-12) {$G_1$ ($k\geq 4$)};
\end{scope}

\begin{scope}[xshift=-37.5cm,yshift=150cm]
\draw[fill=black] (-32.5,-5) {coordinate (v1)} {node[below] {$v_1$}} circle (1.2);
\draw[fill=black] (-12.5,-5) {coordinate (v2)} {node[below] {$v_2$}} circle (1.2);
\draw[fill=black] (12.5,-5) {coordinate (v3)} {node[below] {$v_{k-1}$}} circle (1.2);
\draw[fill=black] (32.5,-5) {coordinate (v4)} {node[below] {$v_k$}} circle (1.2);
\draw[fill=black] (-12.5,30) {coordinate (u1)} circle (1.2);
\draw[fill=black] (12.5,30) {coordinate (u2)} circle (1.2);
\foreach \x in {0,1,...,5} \draw[fill=black] (\x*5-12.5,-5) circle (0.6);
\draw[thick] (u1)--(u2);
\foreach \x in {1,2} \foreach \y in {1,2,3,4} \draw[thick] (u\x)--(v\y);
\node[below] at (0,-12) {$G_2$ ($k\geq 4$)};
\end{scope}

\begin{scope}[xshift=40cm,yshift=150cm]
\draw (0,0) {node {$K_k$}} ellipse (25 and 10);
\draw[fill=black] (-20,0) {coordinate (v1)} circle (1.2);
\draw[fill=black] (20,0) {coordinate (v2)} circle (1.2);
\draw[fill=black] (-20,30) {coordinate (u1)} circle (1.2);
\draw[fill=black] (0,30) {coordinate (u2)} circle (1.2);
\draw[fill=black] (20,30) {coordinate (u3)} circle (1.2);
\foreach \x in {1,2,3} \foreach \y in {1,2} \draw[thick] (u\x)--(v\y);
\node[below] at (0,-12) {$G_3$ ($k\geq 3$)};
\end{scope}

\begin{scope}[xshift=120cm,yshift=150cm]
\draw (0,0) {node {$K_k$}} ellipse (35 and 10);
\draw[fill=black] (-30,0) {coordinate (v1)} circle (1.2);
\draw[fill=black] (-10,0) {coordinate (v2)} circle (1.2);
\draw[fill=black] (10,0) {coordinate (v3)} circle (1.2);
\draw[fill=black] (30,0) {coordinate (v4)} circle (1.2);
\draw[fill=black] (-30,30) {coordinate (u1)} circle (1.2);
\draw[fill=black] (-10,30) {coordinate (u2)} circle (1.2);
\draw[fill=black] (10,30) {coordinate (u3)} circle (1.2);
\draw[fill=black] (30,30) {coordinate (u4)} circle (1.2);
\foreach \x in {1,2} \foreach \y in {1,2} \draw[thick] (u\x)--(v\y);
\foreach \x in {3,4} \foreach \y in {3,4} \draw[thick] (u\x)--(v\y);
\node[below] at (0,-12) {$G_4$ ($k\geq 4$)};
\end{scope}

\begin{scope}[xshift=-120cm,yshift=95cm]
\foreach \x in {1,2,3} \draw[fill=black] (-25,\x*15-30) {coordinate (x\x)} circle (1.2);
\foreach \x in {1,2,3} \draw[fill=black] (25,\x*15-30) {coordinate (y\x)} circle (1.2);
\foreach \x in {1,2,3} \draw[fill=black] (\x*15-30,-25) {coordinate (a\x)} circle (1.2);
\foreach \x in {1,2,3} \draw[fill=black] (\x*15-30,-5) {coordinate (b\x)} circle (1.2);
\foreach \x in {1,2,3} \draw[fill=black] (\x*15-30,5) {coordinate (c\x)} circle (1.2);
\foreach \x in {1,2,3} \draw[fill=black] (\x*15-30,25) {coordinate (d\x)} circle (1.2);
\draw[thick] (x1)--(x3) (y1)--(y3);
\draw[thick] (a1)--(x1)--(b1) (c1)--(x3)--(d1);
\draw[thick] (a3)--(y1)--(b3) (c3)--(y3)--(d3);
\draw[thick] (a1)--(a2) (b1)--(b3) (c1)--(c3) (d1)--(d3);
\foreach \x in {0,1,...,3} \draw[fill=black] (\x*5,-25) circle (0.6);
\node[above] at (-13,-25) {$v_1$}; \node[above] at (a2) {$v_2$}; \node[above] at (13,-25) {$v_k$};
\node[below] at (0,-27) {$G_5$ ($k\geq 3$)};
\end{scope}

\begin{scope}[xshift=-20cm,yshift=95cm]
\foreach \x in {1,2,3} \draw[fill=black] (-55,\x*15-30) {coordinate (x\x)} circle (1.2);
\foreach \x in {1,2,3} \draw[fill=black] (55,\x*15-30) {coordinate (y\x)} circle (1.2);
\draw[fill=black] (-40,-15) {coordinate (z11)} circle (1.2);
\draw[fill=black] (-40,15) {coordinate (z12)} circle (1.2);
\draw[fill=black] (40,-15) {coordinate (z21)} circle (1.2);
\draw[fill=black] (40,15) {coordinate (z22)} circle (1.2);
\foreach \x in {1,2,4,5} \draw[fill=black] (\x*15-45,-25) {coordinate (a\x)} circle (1.2);
\foreach \x in {1,2,...,5} \draw[fill=black] (\x*15-45,-5) {coordinate (b\x)} circle (1.2);
\foreach \x in {1,2,...,5} \draw[fill=black] (\x*15-45,5) {coordinate (c\x)} circle (1.2);
\foreach \x in {1,2,...,5} \draw[fill=black] (\x*15-45,25) {coordinate (d\x)} circle (1.2);
\draw[thick] (z11)--(x1)--(x3)--(z12) (z21)--(y1)--(y3)--(z22);
\draw[thick] (a1)--(z11)--(b1)--(a1) (c1)--(z12)--(d1)--(c1);
\draw[thick] (a5)--(z21)--(b5)--(a5) (c5)--(z22)--(d5)--(c5);
\draw[thick] (a1)--(a2) (a4)--(a5) (b1)--(b5) (c1)--(c5) (d1)--(d5);
\foreach \x in {0,1,...,6} \draw[fill=black] (\x*5-15,-25) circle (0.6);
\node[above] at (-26,-25) {$v_1$}; \node[above] at (a2) {$v_2$}; 
\node[above] at (a4) {$v_{k-1}$}; \node[above] at (26,-25) {$v_k$};
\node[below] at (0,-27) {$G_6$ ($k\geq 5$)};
\end{scope}

\begin{scope}[xshift=65cm,yshift=95cm]
\draw (0,0) {node {$K_k$}} ellipse (10 and 25);
\foreach \x in {1,2,3} \draw[fill=black] (\x*40-40,-20) {coordinate (x\x)} circle (1.2);
\foreach \x in {1,2,3} \draw[fill=black] (\x*40-40,20) {coordinate (y\x)} circle (1.2);
\foreach \x in {1,2,3} \draw[fill=black] (\x*20+10,0) {coordinate (u\x)} circle (1.2);
\draw[fill=black] (20,-10) {coordinate (z11)} circle (1.2);
\draw[fill=black] (20,10) {coordinate (z12)} circle (1.2);
\draw[fill=black] (60,-10) {coordinate (z21)} circle (1.2);
\draw[fill=black] (60,10) {coordinate (z22)} circle (1.2);
\draw[thick] (x1)--(x2)--(z11)--(x1) (y1)--(y2)--(z12)--(y1);
\draw[thick] (x2)--(x3)--(z21)--(x2) (y2)--(y3)--(z22)--(y2);
\draw[thick] (z11)--(z12)--(u1)--(z11) (z21)--(z22)--(u2)--(z21) (x3)--(y3)--(u3)--(x3);
\node[below] at (35,-27) {$G_7$ ($k\geq 3$)};
\end{scope}

\begin{scope}[xshift=-145cm,yshift=20cm]
\draw (0,0) {node {$K_k$}} ellipse (10 and 25);
\foreach \x in {0,2,3,5} 
{\draw[fill=black] (\x*20,-20) {coordinate (a\x)} circle (1.2);
\draw[fill=black] (\x*20,20) {coordinate (e\x)} circle (1.2);}
\foreach \x in {1,4} 
{\draw[fill=black] (\x*20,-10) {coordinate (b\x)} circle (1.2);
\draw[fill=black] (\x*20,10) {coordinate (d\x)} circle (1.2);}
\foreach \x in {1,3,4} \draw[fill=black] (\x*20+10,0) {coordinate (c\x)} circle (1.2);
\draw[thick] (a0)--(a5) (a0)--(b1)--(a2) (a3)--(b4)--(a5);
\draw[thick] (e0)--(e5) (e0)--(d1)--(e2) (e3)--(d4)--(e5);
\draw[thick] (b1)--(c1)--(d1)--(b1) (b4)--(c3)--(d4)--(b4) (a5)--(c4)--(e5)--(a5);
\draw[thick] (a2)--(e2)--(a3)--(e3)--(a2);
\node[below] at (45,-32) {$G_8$ ($k\geq 3$)};
\end{scope}

\begin{scope}[xshift=65cm]
\draw (0,0) ellipse (90 and 10); \node at (10,0) {$K_k$};
\foreach \x in {1,2,...,9} 
{\draw[fill=black] (\x*20-100,0) {coordinate (x\x)} circle (1.2);
\draw[fill=black] (\x*20-100,40) {coordinate (y\x)} circle (1.2);
\draw[fill=black] (\x*20-90,20) {coordinate (z\x)} circle (1.2);
\draw[thick] (x\x)--(y\x)--(z\x)--(x\x);}
\draw[thick] (y1)--(y3) (y4)--(y6) (y7)--(y9);
\foreach \x in {1,4,7} \draw[thick] (y\x) arc (135:45:28.3);
\node[below] at (0,-12) {$G_9$ ($k\geq 9$)};
\end{scope}

\end{tikzpicture}
\caption{Graphs $G_{i},i=1,\cdots,9$.}\label{G12345678}
\end{figure}

\begin{proof}[Proof of Observation~\ref{Th1}]
If $G$ is a 2-connected $P_3$-free graph, then $G$ is complete and clearly has a spanning $\varTheta$-subgraph (if it is not a cycle).
	
Conversely, let $R$ be a connected graph of order at least three such that every 2-connected non-cycle graph contains a spanning $\varTheta$-subgraph. Note that the graphs $G_1$ and $G_6$ are 2-connected and have no spanning $\varTheta$-subgraphs. Hence $R$ must be a common induced subgraph of $G_1$ and $G_6$. From $G_1$, we see that $R$ either is a star or contains an induced copy of $C_4$. However, $G_6$ is $C_4$-free and $K_{1,3}$-free. So $R$ must be a $P_3$.
\end{proof}

We now prove the necessity of Theorem \ref{Th2}.

\begin{proof}[Proof of the necessity of Theorem \ref{Th2}] Let $R,S$ be two connected graphs of order at least 3 other than $P_3$, such that every 2-connected non-cycle $\{R,S\}$-free graph has a spanning $\varTheta$-subgraph. Then each graph $G_i$, $i=1,\ldots,9$, contains $R$ or $S$ as an induced subgraph. We may assume without loss of generality that $G_1$ contains $R$. Note that all the induced subgraphs of $G_1$ with order at least 3 other than $P_3$ either is a star $K_{1,t}$ with $t\geq 3$, or contains an induced $C_4$. It follows that either $R=K_{1,t}$ with $t\geq 3$ or $R$ contains an induced $C_4$.

Suppose first that $R$ contains an induced $C_4$. Since $G_2$, $G_5$ and $G_6$ are all $C_4$-free, they contain $S$ as a common induced subgraph. Since $G_2$ and $G_5$ have no common induced cycles, we see that $S$ is a tree. Since $G_6$ is claw-free, we have that $S$ is a path. But then the largest induced path in $G_2$ is $P_3$ and thus $S=P_3$, a contradiction.

Suppose second that $R$ is a $K_{1,t}$ with $t\geq 5$. Since $G_3$, $G_5$, $G_6$ are all $K_{1,5}$-free, we conclude that they contain $S$ as a common induced subgraph. Since $G_3$ and $G_5$ have no common induced cycles and $G_6$ is claw-free, we have that $S$ is a path. But then the largest induced path in $G_3$ is $P_3$ and thus $S=P_3$, a contradiction.

Suppose thirdly that $R=K_{1,4}$. Since $G_4$, $G_5$, $G_6$ are all $K_{1,4}$-free, they contain $S$ as a common induced subgraph. Then we conclude that $S$ is a path since $G_4$, $G_5$ have no common induced cycles and $G_6$ is claw-free. Notice that the largest induced path in $G_4$ is $P_4$, implying that $S=P_4$.

Finaly we suppose that $R$ is a claw $K_{1,3}$. Since $G_6$, $G_7$, $G_8$, $G_9$ are all claw-free, they contain $S$ as a common induced subgraph. Notice that the largest induced path in $G_9$ is $P_8$. If $S$ is a path, then $S=P_t$ with $t\leq 8$, which is an induced subgraph of $B_{1,5}$, as desired. So we conclude that $S$ contains a cycle. Note that the common induced cycle of $G_6$ and $G_9$ is a triangle. It follows that every induced cycle in $S$ is a triangle. Also note that the length of induced paths connecting any two triangles in $G_6$ is at least four and the length of induced path connecting any two triangles in $G_9$ is at most three. This implies that $S$ contains exactly one triangle. Thus $S$ is a $Z_i$, $B_{i,j}$ or $N_{i,j,k}$. Note that $G_7$ is $\{Z_6,N_{1,1,5},N_{2,2,2}\}$-free, $G_8$ is $B_{3,3}$-free and $G_9$ is $\{B_{2,5},N_{1,2,4}\}$-free. 
It follows that $S$ does not contain each of $\{Z_6,B_{2,5},B_{3,3},N_{1,1,5},N_{1,2,4},N_{2,2,2}\}$ as an induced subgraph. That is, $S$ is an induced subgraph of $B_{1,5}$, $B_{2,4}$, $N_{1,1,4}$ or $N_{1,2,3}$. This proves the necessity of Theorem \ref{Th2}.
\end{proof}

We prove the sufficiency of Theorem \ref{Th2} by reducing to the following two results.

\begin{theorem}\label{Th5}
Let $G$ be a $2$-connected $\{K_{1,4},P_4\}$-free non-cycle graph. Then $G$ has a spanning $\varTheta$-subgraph.
\end{theorem}

\begin{proof}
Let $S$ be a smallest vertex-cut of $G$. By Theorem \ref{Th B}, every vertex in $S$ is adjacent to all vertices in $G-S$. It follows that any independent set of $G$ is either contained in $S$ or in $G-S$. We show that $\alpha(G)\leq 3$. Suppose otherwise that $\alpha(G)\geq 4$. Let $A$ be a maximum independent set of $G$. If $A\subseteq S$, then this together with a vertex in $G-S$, we can get an induced star $K_{1,k}$ with $k\geq 4$; and if $A\subseteq V(G)\backslash S$, then this together with a vertex in $S$, we can get an induced star $K_{1,k}$ with $k\geq 4$, both contradictions. So we conclude that $\alpha(G)\leq 3$. 

If $\kappa(G)\geq 3$, then $G$ is hamiltonian by Theorem \ref{Th D}, and then has a spanning $\varTheta$-subgraph, as desired. So we conclude that $\kappa(G)=2$, i.e., $|S|=2$. Set $S=\{x,y\}$. 

Notice that $G-S$ has at least two component. Since $\alpha(G)\leq 3$, we see that there is a component of $G-S$ that is a clique, say $H$. Now $G-H$ is 2-connected and has independence  number at most 2. It follows that $G-H$ has a Hamilton cycle $C$. Recall that $H$ is a clique and each vertex in $\{x,y\}$ is adjacent to all vertices in $H$. We have that $\langle\{x,y\}\cup V(H)\rangle$ contains a Hamilton path $P$ from $x$ to $y$. Now $C\cup P$ is a spanning $\varTheta$-subgraph of $G$, as desired.
\end{proof}

\begin{theorem}
Let $G$ be a $2$-connected $\{K_{1,3},S\}$-free non-cycle graph, where $S\in\{B_{1,5},B_{2,4},N_{1,1,4},N_{1,2,3}\}$. Then $G$ has a spanning $\varTheta$-subgraph.
\end{theorem}

\begin{proof}
Suppose that $G$ has no spanning $\varTheta$-subgraphs. Since $G$ is claw-free, by Theorem \ref{Th3}, $G$ contains one of graphs in $\mathcal{H}_1\cup\cdots\cup\mathcal{H}_7$ as an induced subgraph. Now we show that each graph in $\mathcal{H}_1\cup\cdots\cup\mathcal{H}_7$ contains an induced copy of $S$. This will complete the proof of the theorem.

Let $F\in\mathcal{H}_1\cup\cdots\cup\mathcal{H}_7$. We label the vertices of $F$ as in Figures \ref{H1234} and \ref{H567}. If $u$ is an end-vertices of a pure link of $F$, then we denote by $\bar{u}$ the neighbor of $u$ which is  an inner vertex of this pure link.
 If $F$ has a pure link $L$ from $u$ to $v$, then we use $P[u,v]$ to denote a shortest path from $u$ to $v$ contained in $L$ (so if $L$ is a path, then $P[u,v]=L$, and if $L$ is a triangle, then $P[u,v]=uv$). Clearly $|P[u,v]|\geq 1$. For $F\in\mathcal{H}_5\cup\mathcal{H}_6\cup\mathcal{H}_7$, there is a chain $H$ from $\{x_1,y_1\}$ to $\{x_2,y_2\}$. We assume without loss of generality that $H$ contains two vertex-disjoint induced path from $x_1$ to $x_2$ and from $y_1$ to $y_2$, which we denote by $Q[x_1,x_2]$ and $Q[y_1,y_2]$. We remark that possibly $|Q[x_1,x_2]|=0$ or $|Q[y_1,y_2]|=0$.

\underline{Case A. $F\in\mathcal{H}_1$.} Let $T$ be the triangle $x_1y_1z_1x_1$, and let $P_1=x_1\bar{x}_1$, $P_2=y_1P[y_1,y_3]y_3x_3$ and $P_3=z_1P[z_1,z_2]z_2y_2P[y_2,y'_2]y'_2z'_3\bar{z}'_3$. Clearly, $P_1$, $P_2$ and $P_3$ are pairwise vertex-disjoint induced paths of $F$. Recall that $|P[y_1,y_3]|\geq 1$, $|P[z_1,z_2]|\geq 1$ and $|P[y_2,y'_2]|\geq 1$. It follows that $|P_1|=1$, $|P_2|\geq 2$ and $|P_3|\geq 5$. This implies that $T\cup P_1\cup P_2\cup P_3$ contains an induced copy of $N_{1,2,5}$, which contains $B_{1,5}$, $B_{2,4}$, $N_{1,1,4}$ and $N_{1,2,3}$. Therefore, $F$ contains $B_{1,5}$, $B_{2,4}$, $N_{1,1,4}$ and $N_{1,2,3}$ as induced subgraphs.

\underline{Case B. $F\in\mathcal{H}_2$.} Let $T=x_1y_1z_1x_1$, and let $P_1=x_1\bar{x}_1$, $P_2=y_1P[y_1,y_3]y_3x_3$ and $P_3=z_1P[z_1,z_2]z_2y_2P[y_2,y'_2]$ $y'_2z'_3\bar{z}'_3$. Similarly as the analysis in Case A, we see that $|P_1|=1$, $|P_2|\geq 2$ and $|P_3|\geq 5$. This implies that $T\cup P_1\cup P_2\cup P_3$ contains an induced copy of $N_{1,2,5}$, which contains induced copies of $B_{1,5}$, $B_{2,4}$, $N_{1,1,4}$ and $N_{1,2,3}$.

\underline{Case C. $F\in\mathcal{H}_3$.} Let $T=x_2y_2z_2x_2$, and let $P_1=y_2\bar{y}_2$, $P_2=x_2P[x_2,x_3]x_3y_3$ and $P_3=z_2P[z_2,z'_2]x'_1P[x'_1,x_1]$ $x_1y_1\bar{y}_1$. Similarly as the analysis in Case A, we see that $|P_1|=1$, $|P_2|\geq 2$ and $|P_3|\geq 5$. This implies that $T\cup P_1\cup P_2\cup P_3$ contains an induced copy of $N_{1,2,5}$, which contains induced copies of $B_{1,5}$, $B_{2,4}$, $N_{1,1,4}$ and $N_{1,2,3}$.

\underline{Case D. $F\in\mathcal{H}_4$.} First we show that $F$ contains induced copies of $B_{1,5}$ and $N_{1,1,4}$. let $T=x_1y_1z_1x_1$, and let $P_1=x_1\bar{x}_1$, $P_2=y_1\bar{y}_1$ and $P_3=z_1P[z_1,z'_1]z'_1x'_2P[x'_2,x_2]x_2y_2\bar{y}_2$. We have that $|P_1|=|P_2|=1$ and $|P_3|\geq 5$. This implies that $T\cup P_1\cup P_2\cup P_3$ contains an induced copy of $N_{1,1,5}$, which contains induced copies of $B_{1,5}$ and $N_{1,1,4}$.

Second we show that $F$ contains an induced copy of $N_{1,2,3}$. Let $T=x'_1x'_2x'_3x'_1$, and let $P_1=x'_1\bar{x}'_1$, $P_2=x'_2P[x'_2,x_2]x_2y_2$ and $P_3=x'_3P[x'_3,x_3]x_3y_3\bar{y}_3$. We have that $|P_1|=1$, $|P_2|\geq 2$ and $|P_3|\geq 3$, and $T\cup P_1\cup P_2\cup P_3$ contains an induced copy of $N_{1,2,3}$. 

Finally we show that $F$ contains an induced copy of $B_{2,4}$. Let $L$ be the pure link of $F$ from $x_1$ to $x'_1$. If $L$ is a triangle, then let $T=L$, $P_1=x_1y_1\bar{y}_1$ and $P_2=x'_1x'_2P[x'_2,x_2]x_2y_2\bar{y}_2$. We have that $|P_1|= 2$, $|P_2|\geq 4$ and $T\cup P_1\cup P_2$ contains an induced copy of $B_{2,4}$. Now assume that $L$ is not a triangle, implying that $|P[x_1,x'_1]|\geq 2$. Let $T=x'_1x'_2x'_3x'_1$, and let $P_1=x'_2P[x'_2,x_2]x_2y_2$ and $P_2=x'_1P[x'_1,x_1]x_1y_1\bar{y}_1$. We have that $|P_1|\geq 2$, $|P_2|\geq 4$ and $T\cup P_1\cup P_2$ contains an induced copy of $B_{2,4}$.

\underline{Case E. $F\in\mathcal{H}_5$.} First we show that $F$ contains an induced copy of $B_{1,5}$. Let $T=a'_1b'_1c'_1a'_1$, $P_1=a'_1\bar{a}'_1$ and $P_2=b'_1P[b'_1,b_1]b_1x_1Q[x_1,x_2]x_2a_2P[a_2,a'_2]a'_2b'_2$. We have that $|P_1|=1$, $|P_2|\geq 5$ and $T\cup P_1\cup P_2$ contains an induced copy of $B_{1,5}$.

Now we show that $F$ contains induced copies of $B_{2,4}$, $N_{1,1,4}$ and $N_{1,2,3}$. Let $T=a_1b_1x_1a_1$, $P_1=a_1\bar{a}_1$, $P_2=b_1P[b_1,b'_1]b'_1c'_1$ and $P_3=x_1Q[x_1,x_2]x_2a_2P[a_2,a'_2]a'_2b'_2\bar{b}'_2$. We have that $|P_1|=1$, $|P_2|\geq 2$, $|P_3|\geq 4$ and $T\cup P_1\cup P_2\cup P_3$ contains an induced copy of $N_{1,2,4}$, which contains induced copies of $B_{2,4}$, $N_{1,1,4}$ and $N_{1,2,3}$.

\underline{Case F. $F\in\mathcal{H}_6$.} First we show that $F$ contains an induced copy of $B_{1,5}$. Let $T=a'_1b'_1c'_1a'_1$, $P_1=a'_1\bar{a}'_1$ and $P_2=b'_1P[b'_1,b_1]b_1x_1Q[x_1,x_2]x_2a_2P[a_2,a'_2]a'_2b'_2$. We have that $|P_1|=1$, $|P_2|\geq 5$ and $T\cup P_1\cup P_2$ contains an induced copy of $B_{1,5}$.

Now we show that $F$ contains induced copies of $B_{2,4}$, $N_{1,1,4}$ and $N_{1,2,3}$. Let $T=a_1b_1x_1a_1$, $P_1=a_1\bar{a}_1$, $P_2=b_1P[b_1,b'_1]b'_1c'_1$ and $P_3=x_1Q[x_1,x_2]x_2a_2P[a_2,a'_2]a'_2b'_2\bar{b}'_2$. We have that $|P_1|=1$, $|P_2|\geq 2$, $|P_3|\geq 4$ and $T\cup P_1\cup P_2\cup P_3$ contains an induced copy of $N_{1,2,4}$, which contains induced copies of $B_{2,4}$, $N_{1,1,4}$ and $N_{1,2,3}$.

\underline{Case G. $F\in\mathcal{H}_7$.} First we show that $F$ contains an induced copy of $N_{1,1,4}$. Let $T=a_1b_1x_1a_1$, $P_1=a_1\bar{a}_1$, $P_2=b_1\bar{b}_1$ and $P_3=x_1Q[x_1,x_2]x_2a_2P[a_2,a'_2]a'_2b'_2\bar{b}'_2$. We have that $|P_1|=|P_2|=1$, $|P_3|\geq 4$ and $T\cup P_1\cup P_2\cup P_3$ contains an induced copy of $N_{1,1,4}$.

Second we show that $F$ contains induced copies of $B_{1,5}$ and $B_{2,4}$. Let $L$ be the pure link from $a_1$ to $a'_1$. If $L$ is a triangle, then let $T=L$, $P_1=a_1b_1\bar{b}_1$ and $P_2=a'_1y_1Q[y_1,y_2]y_2a'_2P[a'_2,a_2]a_2b_2\bar{b}_2$. We have that $|P_1|=2$, $|P_2|\geq 5$ and $T\cup P_1\cup P_2$ contains an induced copy of $B_{2,5}$, which contains induced copies of $B_{1,5}$ and $B_{2,4}$. Assume now that $L$ is not a triangle and thus $|P[a_1,a'_1]|\geq 2$. Let $T=a_2b_2x_2a_2$, $P_1=a_2\bar{a}_2$, $P_2=x_2Q[x_2,x_1]x_1a_1P[a_1,a'_1]a'_1b'_1\bar{b}'_1$. We have that $|P_1|=1$, $|P_2|\geq 5$ and $T\cup P_1\cup P_2$ contains an induced copy of $B_{1,5}$. Again let $T=a_1b_1x_1a_1$, $P_1=P[a_1,a'_1]$, $P_2=x_1Q[x_1,x_2]x_2a_2P[a_2,a'_2]a'_2b'_2\bar{b}'_2$. We have that $|P_1|\geq 2$, $|P_2|\geq 4$ and $T\cup P_1\cup P_2$ contains an induced copy of $B_{2,4}$.

Finally we show $F$ contains an induced copy of $N_{1,2,3}$. 
Recall that $H$ is a chain from $\{x_1,y_1\}$ to $\{x_2,y_2\}$. 
If $H$ is a pure chain (i.e., $k=1$), then it is a bipath chain, implying that there are no edges between $Q[x_1,x_2]$ and $Q[y_1,y_2]$. Let $T=a_1b_1x_1a_1$, $P_1=a_1\bar{a}_1$, $P_2=x_1Q[x_1,x_2]x_2a_2\bar{a}_2$, $P_3=b_1P[b_1,b'_1]b'_1y_1Q[y_1,y_2]y_2b'_2$. We have that $|P_1|=1$, $|P_2|\geq 2$, $|P_3|\geq 3$ and $T\cup P_1\cup P_2\cup P_3$ contains an induced copy of $N_{1,2,3}$. 

Now assume that $H$ is not a pure chain. This implies that $H$ is a chain which is  
consists of at least two pure chains. Let $H_1,H_2$ be the first two pure chains, where $H_1$ is from $\{x_1,y_1\}$ to $\{u_1,v_1\}$ and $\{u_2,v_2\}$ be the original of $H_2$. It follows that $\langle \{u_1,v_1,u_2,v_2\}\rangle$ is a 4-clique. Assume without loss of generality that $u_1,u_2\in V(Q[x_1,x_2])$ and $v_1,v_2\in V(Q[y_1,y_2])$. It follows that $Q[u_1,x_1]$, $Q[u_2,x_2]$ are subpaths of $Q[x_1,x_2]$,
and $Q[v_1,y_1]$ is a subpath of
$Q[y_1,y_2]$. Let $T=u_1v_1u_2u_1$, $P_1=u_1Q[u_1,x_1]x_1a_1$, $P_2=v_1Q[v_1,y_1]y_1b'_1\bar{b}'_1$ and $P_3=u_2Q[u_2,x_2]x_2a_2P[a_2,a'_2]a'_2b'_2$. Recall that $H_1$ is bipath chain. It follows that there are no edges between $P_1$ and $P_2$ other than $u_1v_1$. We have that $|P_1|\geq 1$, $|P_2|\geq 2$, $|P_3|\geq 3$ and $T\cup P_1\cup P_2\cup P_3$ contains an induced copy of $N_{1,2,3}$.
\end{proof}


\begin{thebibliography}{99}
	
	\bibitem{AlFuSa}
	R. Aldred, J. Fujisawa and A. Saito, Forbidden subgraphs and the
	existence of a 2-factor, J. Graph Theory 64 (2010), 250-266.
	
	\bibitem{AlFuSa2}
	R. Aldred, J. Fujisawa and A. Saito, Pairs and triples of forbidden
	subgraphs and the existence of a 2-factor,  J. Graph Theory 90
	(2019), 61-82.
	
\bibitem{91}
P. Bedrossian. Forbidden subgraph and minimum degree conditions for hamiltonicity. Ph. D. Thesis, Memphis State University, USA, 1991.

\bibitem{2}
J.A. Bondy and U.S.R. Murty. Graph Theory. Graduate Texts in Mathematics, Springer, New York, 2008.


\bibitem{B1}J. Brousek, Minimal 2-connected non-hamiltonian claw-free graphs, Discrete Math. 191 (1998), 57-64.


\bibitem{3}
V. Chv\'atal and P. Erd\H{o}s. A note on Hamiltonian circuits. Discrete Math. 2 (1972), 111-113.





\bibitem{4}
Y. Egawa, Proof techniques for factor theorems, in: Horizons of Combinatorics, in: Bolyai Soc. Math. Stud. 17, (Springer, Berlin 2008) 67-78.

\bibitem{5}
R. J. Faudree and R. J. Gould. Characterizing forbidden pairs for Hamiltonian properties. Discrete Math. 173 (1997),  45-60.

\bibitem{FaFaRy}
J. Faudree, R. Faudree and Z. Ryj\'{a}\v{c}ek, Forbidden subgraphs
that imply 2-factors, Discrete Math. 308 (2008), 1571-1582.

\bibitem{FuFuPlSaSc}
J. Fujisawa, S. Fujita, M.D. Plummer and A. Saito and I.
Schiermeyer, A pair of forbidden subgraphs and perfect matchings in
graphs of high connecivity, Combinatorica 31 (2011), 703-723.

\bibitem{FuSa}
J. Fujisawa and A. Saito, A pair of forbidden subgraphs and
2-factor,  Combinatorics, Probability and Computing 21 (2012),
149-158.

\bibitem{FuKaLuOtPlSa}
S. Fujita, K. Kawarabayashi, C.L. Lucchesi, K. Ota, M.D. Plummer and
A. Saito, A pair of forbidden subgraphs and perfect matchings, J.
Combin. Theory Ser. B 96 (2006), 315-324.

\bibitem{HoRyVrWaXi}
P. Holub, Z. Ryj\'a\v{c}ek, P. Vr\'ana, S. Wang and L. Xiong,
Forbidden pairs of disconnected graphs for 2-factor of connected
graphs, J. Graph Theory 100 (2022), 209-231.

\bibitem{lbz}
B. Li, H. Broersma and S. Zhang,
Pairs of forbidden induced subgraphs for homogeneously traceable graphs, Discrete Math. 312 (2012), 2800-2818.

\bibitem{llw}
B. Li, X. Liu and S. Wang, Forbidden pairs of disconnected graphs for supereulerianity of connected graphs, Discrete Math. 347 (2024), 113663.

\bibitem{LiVr}
B. Li and P. Vr\'ana, Forbidden pairs of disconnected graphs
implying hamiltonicity, J. Graph Theory 84 (2017), 249-261.

\bibitem{LvXi}
S. Lv and L. Xiong, Forbidden pairs for spanning (closed) trails,
Discrete Math. 340 (2017), 1012-1018.

\bibitem{LvXi2}
S. Lv and L. Xiong, Erratum to ``Forbidden pairs for spanning (closed) trails"
[Discrete Math. 340 (2017) 1012-1018], Discrete Math. 341 (2018), 1192-1193.

\bibitem{6}
D. Oberly, S. Simic and D. Sumner. Every connected, locally connected non-trivial graph with no induced claw is Hamilton,  J. Graph Theory 3 (1979) 351-356.

\bibitem{OtPlSa}
K. Ota, M.D. Plummer and A. Saito, Forbidden triples for perfect
matchings, J. Graph Theory 67 (2001), 250-269.

\end{thebibliography}
\end{document}